\documentclass[11pt]{amsart} \textwidth=14.5cm \oddsidemargin=1cm
\usepackage{amsmath,amsthm,amssymb,latexsym,epic,bbm,comment,yfonts,mathrsfs}
\usepackage{graphicx,enumerate,stmaryrd,rotating}
\usepackage[all]{xy}
\xyoption{2cell}

\allowdisplaybreaks

\usepackage[active]{srcltx}
\usepackage{youngtab}

\usepackage[latin1]{inputenc}
\usepackage{amsxtra}
\usepackage{amsmath}
\usepackage{amscd}
\usepackage{amssymb}
\usepackage{amsfonts}
\usepackage[all]{xy}
\usepackage{mathrsfs}
\usepackage{amsthm}
\usepackage{color}
\usepackage{upgreek}
\usepackage{verbatim}  
\usepackage{enumerate}
\usepackage{latexsym}
\usepackage{graphicx}
\usepackage{tikz}
\usepackage{todonotes}
\usepackage{setspace}
\usepackage{hyperref}
\usepackage{stmaryrd}
\usepackage{nicefrac}
\usepackage{euscript}
\usetikzlibrary{decorations.pathmorphing}

 \definecolor{darkgreen}{HTML}{336633}
 \definecolor{darkred}{HTML}{993333}
\makeatletter

\newcommand{\arxiv}[1]{\href{http://arxiv.org/abs/#1}{\tt
    arXiv:\nolinkurl{#1}}}

\theoremstyle{plain}
\newtheorem{thm}{Theorem}
\newtheorem*{thm*}{Theorem}
\newtheorem*{thmA}{Theorem A}
\newtheorem*{thmB}{Theorem B}
\newtheorem*{thmC}{Theorem C}

\newtheorem{lem}[thm]{Lemma}
\newtheorem{prop}[thm]{Proposition}

\newtheorem{cor}[thm]{Corollary}
\newtheorem{df-prop}[thm]{Definition-Proposition}

\theoremstyle{definition}

\theoremstyle{remark}
\newtheorem{rem}[thm]{Remark}

\newtheorem*{exs}{Example}

\def\onto{\twoheadrightarrow}

\def\Hom{\operatorname{Hom}\nolimits}

\def\Res{\operatorname{Res}\nolimits}
\def\Ind{\operatorname{Ind}\nolimits}
\def\Ext{\operatorname{Ext}\limits}

\def\gl{\mathfrak{gl}}

\def\la{\lambda}

\def\pn{\mf{pe} (n)}
\def\ov{\overline}

\newcommand{\mc}{\mathcal}
\newcommand{\mf}{\mathfrak}
\newcommand{\C}{\mathbb C}

\newcommand{\oa}{{\bar 0}}
\newcommand{\ob}{{\bar 1}}
\newcommand{\vare}{\epsilon} 

\newcommand{\ad}{\mathrm{ad}}

\newcommand{\fg}{\mathfrak{g}}

\newcommand{\fb}{\mathfrak{b}}

\newcommand{\fh}{\mathfrak{h}}
\newcommand{\fn}{\mathfrak{n}}
\newcommand{\n}{\mathfrak{n}}
\newcommand{\mZ}{\mathbb{Z}}
\newcommand{\cO}{\mathcal{O}}

\newcommand{\mC}{\mathbb{C}}

\newcommand{\h}{\mathfrak{h}}

\newcommand{\ch}{\mathrm{ch}}
\newcommand{\rad}{\mathrm{rad}}

\newcommand{\Coind}{{\rm Coind}}
\newcommand{\coker}{{\rm coker}}

\newcommand{\g}{\mathfrak{g}}
\newcommand{\fl}{\mathfrak{l}}
\newcommand{\fp}{\mathfrak{p}}
\newcommand{\fu}{\mathfrak{u}}
\newcommand{\Real}{\mathrm{Re}}

\newcommand{\cL}{\mathcal{L}}

\newcommand{\Z}{{\mathbb Z}}

 \def\tN{{\widetilde{\mc N}}}

  \newcommand{\add}{{\mathrm{add}}}

     \def\Ann{{\text{Ann}}}
     \def\fd{{\mf d}}
     \def\wfu{{\widetilde{\fd}_\Upsilon}}
       \def\ups{{\Upsilon}}
       \def\Id{{\text{Id}}}
       
              \def\wdL{{\widetilde{L}}}
              \def\wdM{{\widetilde{M}}}
            \def\nbob{{\widetilde{M}^\vee}}
            \def\nbobr{{\widetilde{M}_{\mf b^r}^\vee}}
            
                \def\vpre{{\nu}\text{-pres}}
            
            \def\cvpre{{\nu}\text{-copres}}

\begin{document}

\numberwithin{equation}{section}



\title[Annihilator ideals and blocks of Whittaker modules]{Annihilator ideals and blocks of Whittaker modules over quasireductive Lie superalgebras}

\author{Chih-Whi Chen}
\date{}

\begin{abstract}    
	
  We extend  Kostant's result    on   annihilator ideals of non-singular simple Whittaker modules over Lie algebras to   (possibly singular) simple Whittaker modules over    Lie superalgebras. We describe these annihilator ideals in terms of certain primitive ideals coming from the category $\mc O$ for quasireductive Lie superalgebras. 
		
		To determine these annihilator ideals, we develop   annihilator-preserving equivalences between  certain full subcategories of the Whittaker category $\widetilde{\mc N}$ and the categories of certain projectively presentable modules in the category $\mc O$. These equivalences lead to a   classification of simple Whittaker modules that lie in the integral central blocks when restricted to the even subalgebra. 
	We make a connection between the linkage classes of integral blocks of $\mc O$ and of $\widetilde{\mc N}$. In particular, they can be computed via Kazhdan-Lusztig combinatorics for Lie superalgebras of type $\gl$ and $\mf{osp}$. 		We then give a description of the integral blocks of the category  $\widetilde{\mc N}$ of Whittaker modules for  Lie superalgebras $\gl(m|n)$, $\mf{osp}(2|2n)$ and $\pn$. 
		

	
	
	
	
	  
\end{abstract}

\maketitle

\tableofcontents 


\vspace{5mm}

\section{Introduction}\label{sec1}




\subsection{} In his 1978 seminal paper \cite{Ko78}, Kostant initiated the study of {\em Whittaker modules} over finite-dimensional complex semisimple Lie algebras relative to non-singular characters of the  nil-radicals from triangular decompositions. He developed  various approaches to characterizing simplicity of these modules, including the existence and the  uniqueness  of {\em Whittaker vectors}.  Later on, a systematic construction of (possibly singular) Whittaker modules was studied by    McDowell in \cite{Mc, Mc2}. Building on the work of McDowell and the formulations of the {\em standard Whittaker modules $M(\la,\zeta)$}, a complete classification of simple Whittaker modules was achieved by   Mili{\v{c}}i{\'c} and  Soergel in \cite{MS, MS2}. 
 The   standard Whittaker modules  $M(\la,\zeta)$, parametrized by characters $\zeta$ of nil-radicals and certain coset representatives $\la$ under the action of Weyl group, play the role that Verma modules play  in the classification of  simple highest weight modules. In particular, the definition is set up such that $M(\la,0)$ is the usual Verma module $M(\la)$ of highest weight $\la.$

 Several  achievements have been made in the problem of computing the composition factors of the standard Whittaker modules  by Backelin in \cite{B} and by Mili{\v{c}}i{\'c} and  Soergel in \cite{MS}, where the Kazhdan-Lusztig combinatorics from the BGG category $\mc O$ was introduced to solve the multiplicity problem. 
 Since then, there has been considerable recent progress toward the study of Whittaker modules for Lie algebras, see, e.g., \cite{ALZ, A1, BM, Chri, CoM, R1,Wa2} and references therein.


 \subsection{}	\label{sect::12} Following \cite{Se11}, a finite-dimensional Lie superalgebra $\g=\g_\oa\oplus\g_\ob$ is called {\em quasireductive} if $\mf g_\oa$ is reductive and $\g_\ob$ is a semisimple $\mf g_\oa$-module under the adjoint action, which was also considered in the literature under the name {\em classical}; see, e.g., \cite{Ma}. 
 
  Let $\g=\g_\oa\oplus\g_\ob$ be a   quasireductive Lie superalgebra with a nil-radical $\mf n^+$ coming from a triangular decomposition 
  \begin{align}
  &\mf g=\mf n^-\oplus \mf h \oplus \mf n^+, \label{eq::triag1}
  \end{align} in the sense of \cite[Section 2.4]{Ma}; see also Section \ref{sect::ass}. The {\em simple Whittaker modules} with respect to $\mf n^+$ refer to simple $\mf g$-modules that are locally finite over $\mf n^+$. 

   While the Whittaker modules for Lie algebras have been extensively investigated, the study  of   Whittaker modules over Lie superalgebras is  still at its beginning stage. An earlier result is the construction of   simple Whittaker modules initiated by	 Bagci,  Christodoulopoulou and Wiesner   \cite{BCW}, for Lie superalgebras $\gl(m|n), \mf{sl}(m|n), \mf{psl}(n|n)$ and $ \mf{osp}(2|2n)$.  There are also various other  results done to study Whittaker modules for basic classical Lie superalgebras   and $W$-algebras; see, e.g., \cite{BG2,BGK,Lo09, Lo10, Pr07,Xi,ZS} and references therein.
     
   {
   	In \cite{Ch21}, the author provided a classification of simple Whittaker modules in terms of simple tops of {\em standard Whittaker modules $\wdM(\la,\zeta)$ over Lie superalgebras}. Furthermore, the category $\widetilde{\mc N}$ of Whittaker modules (with respect to $\mf n^+$)   was formulated and studied in \cite{Ch21}, which is a natural generalization of the category $\mc N$. More precisely, the category $\widetilde{\mc N}$ consists of finitely-generated $\g$-modules that are locally finite over the center $Z(\mf g_\oa)$ of $U(\g_\oa)$ and over $\mf n^+$. In particular,    $\widetilde{\mc N}$ contains all simple Whittaker modules over $\g$ (see \cite[Proposition 1]{Ch21}) and every module in $\widetilde{\mc N}$ has finite length (see \cite[Corollary 4]{Ch21}). 

   Let $\ch \mf n^+_\oa:=(\mf n^+_\oa/ [\mf n^+_\oa,\mf n^+_\oa])^\ast$. Let $\Pi_0$ be the simple system for the triangular decomposition 
   	\begin{align}
   	&\g_\oa = \mf n^-_\oa \oplus \h_\oa \oplus \mf n^+_\oa,\label{eq::1111}
   	\end{align} 	the definition of the standard Whittaker modules $\wdM(\la,\zeta)$  from \cite{Ch21} was formulated separately:
   	
   	{\em Case 1}: If $\g$ is a quasireductive Lie superalgebra  of type I, namely, $\g$ admits a compatible $\Z$-gradation $\g =\g_{-1}\oplus \mf g_0 \oplus \g_1$, then  $\wdM(\la,\zeta)$ is defined as $U(\mf g)\otimes_{\mf g_0+\mf g_1}M(\la,\zeta)$, for $\la \in \h^\ast_\oa$ and $\zeta \in \ch \mf n_\oa^+$ by letting $\g_1$ act on $M(\la,\zeta)$ trivially. This type of formulation was first introduced in \cite{BCW} and investigated further in \cite[Section 3.2]{Ch21}. 
   	
   	{\em Case 2}:  Let $\g$ be an arbitrary quasireductive Lie superalgebra. Suppose that the subalgebra $\mf l_\zeta$, generated by $\h_\oa$ and even root vectors $e_\alpha$, $f_\alpha$ of roots $\alpha, -\alpha$ with $\alpha\in \Pi_0$ such that $\zeta(e_\alpha)\neq 0$, is a Levi subalgebra of $\g$ in the sense of \cite[Section 2.4]{Ma}. Then   $\wdM(\la,\zeta)$ is defined as the parabolically induced module from Kostant's simple Whittaker modules in \cite{Ko78}, which we will denote by $Y_\zeta(\la,\zeta)$, associated to the weight $\la\in \h^\ast_\oa$ and the character $\zeta$; see \cite[Section 3.1]{Ch21}.
   	
   	These two formulations of standard Whittaker modules are generalizations of  definitions in \cite{MS}. Also, they coincide in the case when $\g$ is of type I with the Borel subalgebra $\mf b= \mf h +\mf n^+$  chosen such that $\mf b_\ob =\g_1$; see a discussion in \cite[Section 3.2]{Ch21}.  
   	
   	Every standard Whittaker module $\wdM(\la,\zeta)$ has a simple top, denoted by $\wdL(\la,\zeta)$.  For $\g$ a Lie superalgebra satisfying the assumptions in Case 1, it is proved that the following set 
   	\begin{align}
   	&\{\wdL(\la,\zeta)|~\la\in \h_\oa^\ast,~\zeta\in \ch \mf n_\oa^+\},
   	\end{align} is  a complete list of pairwise non-isomorphic simple Whittaker modules; see \cite[Theorem 9]{Ch21}.  For $\g$ a Lie superalgebra with character $\zeta\in\ch \mf n_\oa^+$ satisfying the assumptions in Case 2, the set  
   	\begin{align}
   	&\{\wdL(\la,\zeta)|~\la\in \h^\ast\},
   	\end{align} is  a complete list of pairwise non-isomorphic simple Whittaker modules in the full subcategory  $\widetilde{\mc N}(\zeta)$ of $\widetilde{\mc N}$ consisting of  $M \in \tN$  such that $x-\zeta(x)$  acts locally nilpotently on $M$, for each $x\in \mf n_\oa^+$. The Weyl group $W$, which is defined as the Weyl group of $\g_\oa$ arising from the triangular decomposition \eqref{eq::1111}, has a usual dot-action $\cdot$ on $\h_\oa^\ast$.  In both Case 1 and Case 2, we have 
   	\begin{align}
   	&\wdL(\la,\zeta)\cong \wdL(\mu, \zeta) \Leftrightarrow \mu \in W_\zeta\cdot \la,
   	\end{align} where $W_\zeta$ is the Weyl group of $\mf l_\zeta$. 
   	
   	These constructions provide explicit realizations of  simple Whittaker modules. However, there are some cases left to address; see \cite[Section 3.3]{Ch21} for a detailed discussions.



\subsection{} In \cite{Duflo},  Duflo proved that every primitive ideal of   a reductive Lie algebra is an  annihilator ideal of a simple module in the category $\mc O$. For simple Lie superalgebras, the analogue of Duflo's theorem is established by Musson in \cite{Mu92} using the results of finite ring extensions in \cite{Le89}. Following methods in \cite{Le89,Mu92}, it is shown in \cite[Section 4.1]{CC} that Duflo's theorem remains valid for arbitrary quasireductive Lie superalgebras. Therefore, it is natural to make an identification between
annihilator ideals of simple Whittaker modules and   primitive ideals from the category $\mc O$.

Following \cite{Ko78}, a character $\zeta \in \ch \mf n_\oa^+$ is called {\em non-singular} if $\zeta$ does not vanish on any  root vector associated to a simple root in $\mf n_\oa^+$.  For any $\g$-module $M$,   let $\Ann_{U(\mf g)}M$ denote the annihilator of $M$.  For $\g=\g_\oa$ a reductive Lie algebra, let  $L(\la,\zeta)$ and $L(\la)$ denote the simple quotients of $M(\la,\zeta)$ and $M(\la)$, respectively. Kostant established in \cite[Theorem 3.9]{Ko78} the following result:
\begin{thm*}[Kostant] \label{thm::Kostant}
	Suppose that $\g=\g_\oa$ and $\zeta$ is non-singular. Then, the annihilator ideal of $L(\la,\zeta)$ is given by 
	 \begin{align*} 
	 &{\emph \Ann}_{U(\mf g)}  L(\la, \zeta) = {\emph \Ann}_{U(\mf g)}  L(\la),  
	\end{align*}  for any anti-dominant weight $\la \in \h^\ast$.
\end{thm*}

Since $L(\la,\zeta)\cong L(\la',\zeta)$ for any $\la'\in W_\zeta\cdot \la$, Kostant's result indicates that every annihilator $\Ann_{U(\mf g)}  L(\la, \zeta)$ is a  minimal primitive ideal in the case when $\zeta$ is non-singular. There is a variation of  Kostant's result given by Batra and Mazorchuk  \cite[Section 4.3]{BM}.  One of the main motivations of the present paper is to determine the annihilator ideals of  simple Whittaker modules for quasireductive Lie superalgebras.



 \subsection{} \label{Sect::14} We study several aspects of annihilator ideals of Whittaker modules over quasireductive Lie superalgebras. Namely, the paper attempts to construct  annihilator-preserving equivalences between certain categories of Whittaker modules and  corresponding full subcategories of the BGG category $\mc O$ and to determine annihilator ideals of simple Whittaker modules and to describe the sets of simple Whittaker modules of (indecomposable) blocks for $\widetilde{\mc N}$.

 In the present paper, we will mainly focus on basic classical Lie superalgebras in \eqref{Kaclist1} and quasireductive type-I Lie superalgebras. Unless mentioned otherwise,   we will make two additional assumptions in our setup for a Lie superalgebra $\g$ of type I and the choice of its Borel subalgebra. First, if $\g$ is a quasireductive Lie superalgebra of type I, then we will  assume that its type-I gradation $\g =\g_{-1}\oplus \g_0 \oplus \g_{1}$ is induced by a grading operator from a Cartan subalgebra $\mf h_\oa$ of $\g_\oa$. Next, we will choose a triangular decomposition $\g =\mf n^-\oplus  \h\oplus \mf n^+$ as in \eqref{eq::triag1} such that   the odd subalgebras of $\mf n^\pm$ are $\g_{\pm 1}$, respectively; see  Section \ref{sect::ass}. These assumptions and choice will be referred to as \eqref{ass::A1}, \eqref{ass::A2} and \eqref{pre::tridec} in this article; see also  examples in \eqref{ex::ex1}. Lie superalgebras $\g$ satisfying these assumptions are referred to as {\em Lie superalgebras of type I-0} (see \cite{CC}) and fit into the framework of \cite[Section 4]{CC}. In particular, for such Lie superalgebras there is a number of basic properties of the twisting functors  developed in \cite[Section 4.3]{CC} that are to be used in the present paper.

\subsection{} To start the investigation into the annihilator ideals of Whittaker modules over Lie superalgebras in its full generality, we follow the approach of using Harish-Chandra $(\g,\g_\oa)$-bimodules. This idea goes back at least to the work of Mazorchuk \cite{Ma2} for the  classification of simple modules over the queer Lie superalgebra $\mf q(2)$. Subsequently, this method has been further investigated by  Coulembier, Mazorchuk and the author \cite{CCM} to provide a solution to the problem of classification of simple modules over Lie superalgebras with underlying Lie algebras of type A. 


 \subsection{} \label{sect::166} Before formulating our framework and giving the results, we shall explain our precise
 	setup.
 Let $\g$ be an arbitrary quasireductive Lie superalgebra with a triangular decomposition $\mf g=\mf n^-\oplus \mf h \oplus \mf n^+$ as given in \eqref{eq::triag1}.  We always require additionally that $\h$ is purely even in the paper. We refer to  $\mf b:=\mf h\oplus \mf n^+$ as the corresponding {\em Borel subalgebra}. With slight abuse of notation, we again denote by $\mc O$ the corresponding BGG category.   Denote by $\wdL(\la)$ the simple module in $\mc O$ with highest weight $\la\in \h^\ast$. We fix  $\zeta \in \ch \mf n_\oa^+$ and let $\nu\in \h^\ast$ be dominant such that 
 \begin{align}
 &W_\nu=W_\zeta,
 \end{align} where $W_\nu$ is the stabilizer group of $\nu$ under the dot-action of $W$. Let $\mc X\subset \mf h^\ast$  denote the set of all integral weights.   Let $\mc X(\nu)$ denote the set of weights $\mu\in \nu+\mc X$ satisfying  that $\wdL(\mu)$ is $\alpha$-free (i.e. the root space  $\g_\oa^{-\alpha}$ acts on $\wdL(\mu)$ freely), for any simple root $\alpha$ of $\mf l_\zeta$; see  Section \ref{sect::24f} and Section \ref{sect::411}.
   

 Let $S$ be an arbitrary simple Whittaker module.  By \cite[Proposition 1]{Ch21}, there exists a simple  Whittaker $\g_\oa$-submodule $V$ of $S$. We refer to $S$ as a {\em simple Whittaker module of integral type}  provided that $V$ admits a central character of $\g_\oa$ associated with an integral weight. This definition of integrability for a simple Whittaker module $S$ is independent of the choice of $V$; see Proposition \ref{prop::intWhi}.

  We need the following preparatory notations and conventions before giving the first main result.

\begin{itemize}
	\item[$\bullet$]  Let $\mc O^{\vpre}$ be the category of {\em projectively presentable} modules of $\mc O$ in the sense of \cite{KoM,MaSt04}, namely, 
	$\mc O^{\vpre}$ is the full subcategory of $\mc O$, which consists of all modules $M$ that has a presentation $P_1\rightarrow P_2\rightarrow M \rightarrow 0$ by projective modules $P_1, P_2$ such that every simple quotient of $P_1$, $P_2$ has a highest weight lying in $\mc X(\nu)$. The  simple objects $S_\nu(\mu)$ in $\mc O^{\vpre}$ are indexed by $\mu\in \mc X(\nu)$; see also Section \ref{Sect::35}. 
	\item[$\bullet$] Let $\mc B_\nu$ be the category of {\em Harish-Chandra $(\mf g,\mf g_\oa)$-bimodules} $Y$ such that $Y\text{Ker}\chi_\nu^\oa=0$ as considered in \cite[Section 4.3]{Ch21}, where $\text{Ker}\chi_\nu^\oa$ denotes the kernel of the central character $\chi_\nu^{\oa}$ of $\g_\oa$ associated with the weight $\nu$; see also Section \ref{Sect::HCmod}. 
	\item[$\bullet$] Let $\widetilde{\mc N}(\nu+\mc X, \zeta)^1$ be the  category consisting of Whittaker modules of the form $Y\otimes_{U(\mf g_\oa)}M(\nu,\zeta)$ with $Y\in \mc B_\nu$. See also Section \ref{sect::412}  for an intrinsic definition.  
		\item[$\bullet$] For any $\g_\oa$-module $M$ and  $\g$-module $N$, let  $\mc L(M,N)$ denote the maximal $(\g, \g_\oa)$-submodule of $\Hom_\C(M,N)$ that is a  direct sum of finite-dimensional $\g_\oa$-modules under the adjoint action of $\g_\oa$ given by $x\cdot f=xf-fx$ for any $x\in \g_\oa$ and $f\in \Hom_\C(M,N)$. Let $\g$-Mod (resp. $\g_\oa\mbox{-Mod}$) denote the category of $\g$-modules (resp. $\g_\oa$-modules). Then this defines a functor $\mc L(-,-)$ from $\g_\oa\mbox{-Mod}\times {\g} \mbox{-Mod}$ to $\g\otimes_{\mC}  \g_\oa\mbox{-Mod}$ (i.e., the category of $\g\otimes_{\mC}  \g_\oa$-modules).
\end{itemize}
Let $\Lambda^+(\zeta):=\{\nu\in \h^\ast/\mc X|~\nu\text{ is dominant such that }W_\nu=W_\zeta \}$, namely, $\Lambda^+(\zeta)$ is a complete list of representatives in $\mf h^\ast/\mc X$ that has the stabilizer group $W_\nu$. Our first main result is the following classification of  simple Whittaker modules of integral type and their annihilator ideals. 
\begin{thmA} For each $\nu \in \Lambda^+(\zeta)$, 
	we have an annihilator-preserving equivalence of categories 
	\begin{align}
	&\mc L(M(\nu),-)\otimes_{U(\g_\oa)} M(\nu,\zeta): \mc O^{\vpre}\cong \mc B_\nu\cong \widetilde{\mc N}(\nu+\mc X, \zeta)^1. \label{eq::11}
	\end{align}
 We define $S(\mu,\zeta)$ as the image of $S_\nu(\mu)$ under the equivalence \eqref{eq::11}, namely, 
\begin{align}
&S(\mu,\zeta):=\mc L(M(\nu), S_\nu(\mu))\otimes_{U(\g_\oa)} M(\nu,\zeta), \text{ for  $\mu \in \mc X(\nu)$.}
\end{align}  Then 
\begin{itemize}
	\item[(1)] The set 
	\begin{align}
	&\{S(\mu,\zeta)|~\zeta\in \ch \mf n_\oa^+,~\nu\in \Lambda^+(\zeta),~\mu \in \mc X(\nu)\},
	\end{align}
	is a list  of pairwise non-isomorphic simple Whittaker modules over $\g$. 
	In particular, the set 
	\begin{align}
	&\{S(\mu,\zeta)|~\zeta\in \ch \mf n_\oa^+,~\nu\in \Lambda^+(\zeta)\cap \mc X,~\mu \in \mc X(\nu)\},
	\end{align}
   is an  exhaustive list  of (pairwise non-isomorphic)  simple Whittaker modules of integral type over $\g$.
	\item[(2)] Furthermore, assume that $\g$ is either a basic classical Lie superalgebra  \eqref{Kaclist1} or a Lie superalgebra of type I-0 (i.e.,  the triangular decomposition \eqref{eq::triag1} of $\g$ satisfies the assumptions \eqref{ass::A1}-\eqref{pre::tridec}) with $\nu\in \mc X$. Then $S_\nu(\mu)$ can be realized as images of simple highest weight modules under the Arkhipov's twisting functor in the sense of \cite{AS, CMW}, and we have 
	\begin{align}
	&{\emph \Ann}_{U(\mf g)} S(\mu,\zeta) = {\emph \Ann}_{U(\mf g)} \widetilde L(\mu),
	\end{align} 
	for any $\mu \in \mc X(\nu)$.
\end{itemize}	
\end{thmA}

 As mentioned in Section \ref{sect::12}, the problem of  constructing of simple Whittaker modules remains open for the cases of non-type I Lie superalgebras with $\zeta\in \ch \mf n_\oa^+$ such that $\mf l_\zeta$ is not a Levi subalgebra of $\g$. Therefore, the advantage of the construction for $S(\mu,\zeta)$ given in Theorem A is to provide a complete classification   of integral type simple Whittaker modules in its full generality.  However, this approach gives less information about the  module structure, for instance, it is less clear how to compute the basis for  $S(\mu,\zeta)$ explicitly; see also  \cite[Section 5.3.2]{Ch21} for concrete examples of simple Whittaker modules over $\mf{gl}(1|2)$. We will identify $S(\la, \zeta)$ as $\wdL(\la,\zeta)$ for   Lie superalgebras $\g$ of type I-0; see Theorem \ref{thm::15}.




\subsection{} \label{sect::16} 

 A weight is said to be {\em $W_\zeta$-anti-dominant} if it is anti-dominant as a weight of $\mf l_\zeta$. We set $\mc O^\Z$ to be the full subcategory of $\mc O$ consisting of modules with integral weights.   	Our second main result is the following, which  generalizes Kostant's result. 





\begin{thmB} \label{thm::thmA} Suppose that $\g$ is a Lie superalgebra of type I-0 with $\zeta\in \ch \mf n_\oa^+$. Then we have 	\begin{align*} 	&{\emph \Ann}_{U(\mf g)} \widetilde L(\la, \zeta) = {\emph \Ann}_{U(\mf g)} \widetilde L(\la), 	\end{align*}
  for any	 $W_\zeta$-anti-dominant and integral weight $\la \in \h^\ast$. 	
\end{thmB} 
 
 In the case that $\la$ is typical and $M(\la,\zeta)$ is simple (resp. anti-dominant), Theorem B recovers \cite[Corollary 29]{Ch21} (resp. \cite[Corollary 31]{Ch21}) for the Lie superalgebras $\gl(m|n), \mf{osp}(2|2n)$ (resp. $\pn$); see also \cite[Section 5.5]{Ch21} for the idea of using Harish-Chandra $(\g,\g_\oa)$-bimodules. In the case that $\g=\g_\oa$ and $\la$ is anti-dominant, Theorem B recovers \cite[Theorem 39]{BM}; see also Section \ref{Sect::cons} for more consequences of Theorem B. 

For reductive Lie algebras $\g=\g_\oa$ and Lie superalgebras $\g=\gl(m|n)$, $\mf{osp}(2|2n)$, the inclusions between primitive ideals have been studied in literature; see, e.g., \cite{Co16,CoM1,Ja83,Le96,Vo80} and references therein. In particular, one can use Theorem B  to determine all inclusions between the annihilator ideals of simple Whittaker modules in terms of the $\Ext^1$-quiver of the category $\mc O$ in the case of reductive Lie algebras and $\gl(m|n)$ (see  \cite[Theorem 7.4]{Co16}). As a consequence, Theorem B implies the validity of the conjecture of Batra and Mazorchuk given in \cite[Conjecture 40]{BM} for the special case of simple Whittaker modules of integral type; see Section \ref{sec::533}.


  	Let us briefly explain our idea of the proofs of the conclusions concerning Lie superalgebras of type I-0 in Theorems A and B below. We shall show that the equivalence in \eqref{eq::11} sends  $S_\nu(\mu)$  to  $\wdL(\mu,\zeta)$, for any integral $W_\zeta$-anti-dominant weight $\mu$, namely, we have $S(\mu,\zeta)=\wdL(\mu,\zeta)$.
  	 Therefore, the equivalence \eqref{eq::11} allows to reduce the proof of Theorem B to the description of annihilator ideals  $\Ann_{\text{U}(\g)}S_\nu(\mu)$ in Part (2) of Theorem A.
  	  

  	 
  	To determine $\Ann_{U(\g)}S_\nu(\mu)$, we recall that there is a usual approach to the study of primitive ideals using   the Joseph's  completion functors and the Arkhipov's twisting functors; see, e.g., \cite{AS,CoM1, KM}. Following \cite{KM}, we realize twisting functors  on the category $\mc O^\Z$  as partial coapproximation functors. This then enables us to prove that each simple object $S_\nu(\mu)$ in $\mc O^{\vpre}$ is isomorphic to a twisted simple module in $\mc O$ and to identify  $\Ann_{U(\g)} S_\nu(\mu)$ with  $\Ann_{U(\g)} \wdL(\mu)$, for any integral and  $W_\zeta$-anti-dominant weight $\mu$.

\subsection{} As a   subcategory of $\widetilde{\mc N}$, the category $\mc O^\Z\cong \widetilde{\mc N}(\mc X,0)^1$ has a block decomposition determined by the central characters of $\g$; see  \cite[Section 1.13]{Hu08} for reductive Lie algebras and \cite{CMW, Se03} for Lie superalgebras. Denote by  $\widetilde{\mc N}^\Z$ the Serre subcategory of $\widetilde{\mc N}$ generated by  simple Whittaker modules of integral type. We  consider linkage classes for (indecomposable) blocks of $\widetilde{\mc N}^\Z$.  
	
	 Our main examples are  $\g = \gl(m|n),\mf{osp}(2|2n)$ and $\pn$. In each of these cases,  $\widetilde{\mc N}^\Z$ is generated by $\wdL(\la,\zeta)$ with $\la\in \mc X$ and $\zeta \in \ch \mf n_\oa^+$. It follows from Theorem B  that $\wdL(\la,\zeta)$ and $\wdL(\la)$ admit the same central character of $\g$ provided $\la$ is $W_\zeta$-anti-dominant. In the cases that $\g =\gl(m|n), \mf{osp}(2|2n)$, we will show that the  $\widetilde{\mc N}^\Z$ is a direct  summand of $\widetilde{\mc N}$ and its  blocks can be described via the central characters just as for the category $\mc O^\Z$. 

More precisely, we fix   a $W$-invariant bilinear form $\langle,\rangle$ on $\h^\ast$ as in Sections \ref{sect::271}-\ref{sect::272}. Denote the Weyl vector by $\rho$ (see Section \ref{Sect::224}) and  define an equivalence $\sim$ on $\mc X$  by declaring  
$$\la\sim \mu,  \text{ for }\la,\mu \in \mc X,\text{ if } \mu = w \cdot (\la-\sum_{i=1}^kc_i \alpha_i),$$ for some $w\in W$,  integers $$c_1, c_2,\ldots, c_k$$ and mutually orthogonal isotropic roots $$\alpha_1, \alpha_2,\ldots, \alpha_k$$  such that $\langle\la+\rho, \alpha_i\rangle=0,$ for all $1\leq i\leq k.$  The following is our
third main result:

\begin{thmC} Let $\nu \in \mc X$ be dominant such that $W_\nu =W_\zeta$. For each $\la\in \mc X(\nu)$, denote the projective cover of $S(\la,\zeta)$ in $\widetilde{\mc N}(\mc X, \zeta)^1$ by $\widetilde{P}(\la,\zeta)$.
	Then 
	\begin{itemize}
		\item[(1)] {\em(}Cartan matrix{\em)}  Suppose that $\g$ is either a  basic classical Lie superalgebra or a Lie superalgebra of type I-0. For any $\la,\mu\in \mc X(\nu)$, we have the multiplicity formula \begin{align}
		&[\widetilde P(\la,\zeta):S(\mu,\zeta)] = [\widetilde P(\la): \wdL(\mu)], \label{eq::mul}
		\end{align} where $\widetilde P(\la)$ denotes the projective cover of $\wdL(\la)$ in $ {\mc O}$. If particular, if the multiplicity in \eqref{eq::mul} is positive then $S(\la,\zeta)$ and $S(\mu,\zeta)$ lie in the same block of $\widetilde{\mc N}$.
		\item[(2)] 
		Let $\g=\gl(m|n),~\mf{osp}(2|2n)$ with $\zeta \in \ch\mf n_\oa^+$.  For any integral weights $\la, \mu \in \h^\ast$,  the simple modules $\widetilde{L}(\la,\zeta)$ and $\widetilde{L}(\mu,\zeta)$ are in the same  block of $\widetilde{\mc N}$ if and only if $\la \sim \mu$.
		\item[(3)] 		For $\mf g = \pn$,   the    $\widetilde{\mc N}^\Z$ has at most $n+1$ blocks, up to equivalence.
	\end{itemize}
\end{thmC}
We remark that   Part (1) of Theorem C provides a full   description of blocks of the category $\widetilde{\mc N}(\mc X, \zeta)^1$, generalizing that of  $\widetilde{\mc O}^\Z \cong \widetilde{\mc N}(\mc X,0)^1$. In particular, in the case of ortho-symplectic Lie superalgebras $\mf{osp}(2m|2n+1)$ and $\mf{osp}(2m|2n)$ beyond type I, the multiplicity \eqref{eq::mul} can also be computed via Kazhdan-Lusztig combinatorics by the work of Bao and Wang \cite{BaoWang, BaoWang2}.

\subsection{}   The paper is organised as follows. In Section \ref{Sect::Pre}, we provide some background materials on quasireductive Lie superalgebras. In particular, we review the parabolic decompositions, the BGG category $\mc O$, the dualities and the   Lie superalgebras of type I-0.  

In Section \ref{Sect::Realization}, several results on Harish-Chandra bimodules are introduced. We set up the realizations of Joseph's Enright completion functors and  Arkhipov's twisting functors as partial approximation functors and partial coapproximation functors, respectively.  Section \ref{Sect::35} offers a description of the category $\mc O^{\vpre}$, for any dominant  weight $\nu\in \h^\ast$.  We provide a realization of simple objects $S_\nu(\mu)$ in $\mc O^{\vpre}$ in terms of twisted simple modules in the category $\mc O$. 

Section \ref{Sect::AnnSWhi} of the paper is devoted to the proofs of Theorem A and Theorem B.  The necessary preliminaries for Whittaker modules are gathered in Section \ref{Sect::41}.  We establish in Section \ref{Sect::PfOfThmA} the equivalence \eqref{eq::11} introduced in Theorem A. We then bring all the above together to obtain   complete proofs of Theorems A and  B.  The proof of  Part (1) of Theorem C is given in Section \ref{Sect::cons}. Section \ref{Sect::5Link}  is devoted to some consequences and the proof of Part (2) of Theorem C.  Some helpful  technical tools are developed in Section \ref{Sect::Tetoolblocks}. Finally, in Section \ref{cor::644} we give an example of a block of $\widetilde{\mc N}$ for the non-type I Lie superalgebra $\mf{osp}(3|2)$.




\subsection*{Acknowledgment}  The author is partially supported by a MoST grant and the National Center for Theoretical Sciences, and he would like to deeply thank  Shun-Jen Cheng and  Volodymyr Mazorchuk for helpful discussions and  comments.

\section{Preliminaries} \label{Sect::Pre}
Throughout the paper, the symbols $\C,\mathbb R, \Z, \Z_{\geq 0}$ stand for the sets of all complex numbers, real numbers, integers and non-negative integers. All vectors spaces, algebras, tensor products, et cetera, over the field $\C$. Denote by $\Z_2=\{\oa, \ob\}$ the abelian group of two elements. For a homogeneous element $x$ of a vector superspace $V=V_\oa\oplus V_\ob$, we denote its parity by $\overline x\in \Z_2$, namely, $\ov x=i$ if $x\in V_i$, for $i\in \oa, \ob$. In this section, we collect preliminaries on quasireductive Lie superalgebras which are relevant
for the rest of the paper.

\subsection{Lie superalgebras and functors between representation categories} \label{sect::pre1} Let $\g=\g_\oa\oplus \g_\ob$ be a finite-dimensional Lie superalgebra, see, for example, \cite{Ka1}. Throughout the present paper, we assume that $\g$ is a quasireductive Lie superalgebra, namely,   $\mf g_\oa$ is reductive and $\g_\ob$ is a semisimple $\mf g_\oa$ under the adjoint action; see \cite{Se11}. We will mainly consider  the following series of quasireductive Lie superalgebras
from Kac's list in \cite{Ka1}:
\begin{align} 
&\gl(m|n),~\mf{sl}(m|n),~\mf{psl}(n|n),~\mf{osp}(m|2n),~D(2,1|\alpha),G(3),F(4),\label{Kaclist1} \\ 
&\mf{p}(n),~[\mf{p}(n),\mf{p}(n)],~\mf{q}(n),~\mf{sq}(n),~\mf{pq}(n),~\mf{psq}(n).\label{Kaclist2}
\end{align}  We refer to \cite{ChWa12,Mu12} for more details. A Lie superalgebra   in \eqref{Kaclist1} is called {\em basic classical}; see \cite[Section 1.1]{ChWa12}.  We denote the universal enveloping algebra by  $U(\g)$ and its center by $Z(\g)$. 
We will sometimes use the notations $\widetilde U: = U(\mf g)$ and $U:=U(\mf g_\oa)$.

	We denote by $\mf g$-Mod  the category of all $\mf g$-supermodules, with parity preserving module morphisms. The parity shift functor on $\g$-Mod is denoted by $\mc S$. Similarly, we define the category $\mf g_\oa$-Mod. Observe that $\mf g_\oa$-Mod is the direct sum of two copies of the usual representation category.

Fo any $M\in \g$-Mod, the socle $\text{soc}(M)$ of $M$ is defined as the sum of all simple submodules, in case $M$ has a simple submodule, and as zero otherwise. Dually, the radical $\rad(M)$ is defined as the intersection all maximal submodules, in case $M$ has a maximal submodule. We define $\rad(M)=M$ in case $M$ has no maximal submodule. The top of   $M$ is defined as $\text{top}(M):=M/\rad M$.

For any subalgebra $\mf s$ of $\g$, we denote by $\Res^{\mf g}_{\mf s}(-)$ the restriction functor from $\mf g$-Mod to $\mf s$-Mod. Then $\Res^{\mf g}_{\mf s}(-)$ has left and right adjoint functors
$$\Ind^{\g}_{\mf s}(-)=U(\g)\otimes_{U(\mf s)}-\qquad\mbox{and}\qquad \Coind^{\g}_{\mf s}(-)=\Hom_{U(\mf s)}(U(\g),-).$$
We let $\Res(-) = \Res_{\mf \g_\oa}^{\g}(-): \mf g\text{-Mod} \rightarrow \mf g_\oa\text{-Mod}$ denote the restriction functor from $\mf g$ to $\mf g_\oa$. Also, we have exact induction and coinduction functors $\Ind(-):=\Ind_{\g_\oa}^{\g}(-)$ and $ \Coind(-):=\Coind_{\g_\oa}^{\g}(-)$ from $\mf g_\oa\text{-Mod}$ to $\mf g\text{-Mod}$. 
They are left and right adjoint functors to $\Res(-)$. 
By \cite[Theorem~2.2]{BF} (see also \cite{Go}) we have isomorphism of functors
$\Ind(-)\;\cong\; \Coind(\Lambda^{\text{max}}(\mf g_\ob)\otimes -).$

For any $\g$-module $M$ having a composition series, we will freely use   $[M:L]$ to denote the Jordan-H\"older decomposition multiplicities of a simple module $L$ in a composition series of $M$. 

\subsection{Parabolic and triangular decompositions}
In this section, we follow the convention of, e.g., \cite[Section 2.4]{Ma} and \cite[Section 1.4]{CCC}, to define parabolic and triangular decompositions in the paper. 
\subsubsection{} \label{sect::221}
Throughout the present paper, we fix a triangular decomposition of $\g_\oa$
\begin{align}
&\g_\oa = \mf n_\oa^- \oplus \mf h \oplus \mf n_\oa^+. \label{eq::parad}
\end{align} with Cartan subalgebra $\mf h$ and nil-radials $\mf n^\pm_\oa$. The triangular  decomposition   in \eqref{eq::parad} gives rise to the set $\Phi\subset \h^\ast$ of roots  and the associated {\em root space decomposition}:
\begin{align*}
&\mf g =\bigoplus_{\alpha \in \Phi\cup \{0\}} \mf g^\alpha,
\end{align*} where $\mf g^\alpha :=\{X\in \g|~[h,X] = \alpha(h)X, \text{~for all }h\in \h\}$.

We recall the definitions of triangular decomposition, parabolic decomposition and Borel subalgebra of the quasireductive Lie superalgebra $\g$ from \cite[Section 1.4]{CCC} as follows. For a given $h\in\fh$, the  {\em parabolic decomposition} $\fg=\fu^-\oplus\fl\oplus\fu^+$ (determined by the vector $h$) is given by 
\begin{equation}\label{deflu}\fl:=\bigoplus_{\Real \alpha(h)=0} \fg^\alpha,\quad \fu^+:=\bigoplus_{\Real \alpha(h)>0} \fg^\alpha, \quad \fu^-:=\bigoplus_{\Real \alpha(h)<0} \fg^\alpha,\end{equation} where $\Real(z)$ denotes the real part of $z\in \mathbb C$. 

In the paper, we consider the Borel-Penkov-Serganova subalgebras from  \cite{PS} (see also \cite[Section 3.2]{Mu12} and \cite[Section 1.4]{CCC}) with  one additional assumption that {\em the subalgebra $\g^0$ is purely even}. Namely, the parabolic decomposition in \eqref{deflu} is said to be a {\em triangular decomposition} if it satisfies \begin{align}
&\mf l = \mf g^0 =\mf h.
\end{align}   In this case, we let $\mf n^\pm:=\mf u^\pm$ and we refer to   $\mf b:= \mf h\oplus \mf n^+$ as    the  corresponding {\em Borel subalgebra} of $\g$. We remark that this assumption excludes $\g$ to be in the family of Q-type Lie superalgebras in \eqref{Kaclist2}.  We also  remark that in \cite[Section 1.4]{CCC} the Cartan subalgebra is defined as $\mf g^0$ and not necessarily even in \cite{CC}.  For a given Borel subalgebra $\bf b$ of $\mf g_\oa$, every Borel subalgebra of $\g$ is conjugate to one which has $\bf b$ as underlying even subalgebra by \cite[Section 1.3]{CCC}.

Let $\Phi^+$ be the set of roots of $\mf n^+$.  Denote by $\Phi^+_\oa\subseteq \Phi^+$  the set of  roots of $\mf n^+_\oa$, which forms a positive system of $\g_\oa$.  We set $\Pi_0$ to be    the corresponding simple system for $\Phi^+_\oa$.

\subsubsection{}  \label{Sect::224} The Weyl group $W$ is defined as the Weyl group of $\g_\oa$ with respect to the triangular decomposition \eqref{eq::parad}. We fix a $W$-invariant bilinear form $\langle ,\rangle$ on $\h^\ast$. Define  $$\rho_\oa:= \frac{1}{2}\sum_{\alpha\in \Phi_\oa^+}\alpha,~\rho_\ob:=\frac{1}{2}\sum_{\alpha \in \Phi^+\backslash \Phi_\oa^+}\alpha.$$ The {\em Weyl vector} $\rho \in \h^\ast$ is defined by $\rho:=\rho_\oa -\rho_\ob$. For $\alpha \in \Phi_\oa^+$, we let $s_\alpha$ denote the reflection associated with the   root $\alpha$. 
The corresponding dot-action of $W$ on $\h^\ast$ is defined  as \begin{align}
&w\cdot \la =w(\la+\rho_\oa) -\rho_\oa, \label{eq::dotact}
\end{align} for any $\la \in \h^\ast$. For any $\alpha \in \Phi_\oa^+$, we set $\alpha^\vee:= 2\alpha/\langle\alpha, \alpha\rangle$; see \cite[Section 0.2]{Hu08}. A weight is said to be {\em integral} if $\langle \la, \alpha^\vee\rangle \in \Z$, for any $\alpha\in \Phi_\oa^+$. 
We denote by $\mc X \subset\fh^\ast$ the set of integral weights.  A weight $\la$ is said to be  {\em dominant} (resp. {\em anti-dominant}) if $\langle \la+\rho_\oa, \alpha^\vee \rangle \notin \Z_{<0}$ (resp. $\langle \la+\rho_\oa, \alpha^\vee \rangle \notin \Z_{>0}$), for any $\alpha \in \Phi_\oa^+$. For a given   weight $\la$, we let $W_\la$ denote the stabilizer subgroup of $\la$ under the dot-action of $W$. Finally, we let $w_0\in W$ denote the longest element in $W$. 

\subsubsection{} \label{Sect::222}  As observed in \cite[Section 3.1]{Mu12} and \cite[Section 1.3]{CCC}, for any $x\in \g_\oa$ that is nilpoent in $\g_\oa$, the automorphism $\exp(\ad x)$ of $\g_\oa$ extends to an automorphism of $\g$. For each $w\in W$, we extend the action of $w$ on $\g_\oa$ to an automorphism $\varphi^w$ of $\g$ in this way. In particular, we recall  the automorphism $\varphi^{w_0}$    associated to  $w_0$  from \cite[Section 1.3]{CCC}, which  defines an  anti-involution $\sigma$ on $\g$ by letting $$\sigma :=-\varphi^{w_0}.$$
The $\sigma$ satisfies that 
\begin{align}  
&\sigma(\mf h)= \mf h,~\sigma(\mf n^+_\oa) =\mf n_\oa^-, \label{eq::gdinv} \\
&(\widehat \la)(h) = \la(\sigma(h)),\text{ for } h\in \h \text{ and }\la \in \h^\ast, 
\end{align} where $\widehat{(\cdot)}:\h^\ast \rightarrow \h^\ast$ denotes  the involution given by $\widehat \la =-w_0\la$.  We remark that an anti-involution satisfying \eqref{eq::gdinv} is  also referred to as a {\em good involution} in the sense of \cite[Section 2.2.4]{CC}. 

If $\g$ is a basic classical Lie superalgebra from \eqref{Kaclist1}, there is also a natural antiautomorphism  introduced in \cite[Proposition 8.1.6]{Mu12}, which we will denote by $\sigma'$. In this case, we have $\la(h) = \la(\sigma'(h))$, for any $h\in \h$.

\subsection{Lie superalgebras of type I-0} \label{sect::ass}

In this paper, when the statements or proofs are specific to Lie superalgebras of type I, we will make additional assumptions \eqref{ass::A2} on $\g$ and \eqref{pre::tridec} on its triangular decomposition, which we shall explain in this section.

 \subsubsection{} 
 A (quasireductive) Lie superalgebra $\g$  is referred to as a {\em Lie superalgebra of type I-0} if $\g$ is equipped with a compatible $\Z$-grading $\g =\g_{-1}\oplus \g_0 \oplus \g_{1}$ induced by a grading element $H\in \h$, that is, 
 \begin{align} 
 &\g_0 =\g_\oa \text{ and } \g_\ob =\g_{-1}\oplus\g_1 \text{ with }[\g_1,\g_1] =[\g_{-1},\g_{-1}] =0. \label{ass::A1}\tag{{\bf A1}}\\
 &[H, x]= kx, \text{ for }x \in \g_k \text{ with }k=\pm 1, 0. \label{ass::A2}\tag{{\bf A2}}
 \end{align}  We will use   notations $\g_{\leq 0}:=\g_0\oplus \g_{-1}$ and  $\g_{\geq 0}:=\g_0\oplus \g_{1}$.
 
 \subsubsection{} \label{Sect::2231}  In the case of Lie superalgebras of type I-0, we always assume that the associated triangular decompositions satisfy the assumption \eqref{pre::tridec} as follows.  Fix a triangular decomposition  \eqref{eq::parad}  with $\mf n^{\pm} :=   \mf n^{\pm}_\oa \oplus \g_{\pm1}$. We claim that \eqref{ass::A1} and \eqref{ass::A2} imply the following assertion  
 \begin{align} 
 &\mf g= \mf n^- \oplus \mf h \oplus \mf n^+ \text{ is a triangular decomposition of $\g$}. \label{pre::tridec} \tag{{\bf A3}}
 \end{align} 
   To see this, suppose that the triangular decomposition \eqref{eq::parad} is determined by $h'\in \h$, namely, $\mf h=\bigoplus_{\Real \alpha(h')=0} \fg^\alpha_\oa,\quad \fn^+_\oa=\bigoplus_{\Real \alpha(h')>0} \fg_\oa^\alpha$ and $ \fn_\oa^-=\bigoplus_{\Real \alpha(h')<0} \fg_\oa^\alpha$. Then there is a positive real number $\epsilon$ such that $H+\epsilon h'$ determines the desired triangular decomposition \eqref{pre::tridec}.  We   call $\mf b:= \mf h+\mf n^+$ the {\em standard Borel subalgebra}. By \cite[Section 1.4]{CCC}, the subalgebra $\mf b^r:= \mf b_\oa +\g_{-1}$ is again a Borel-Penkov-Serganova subalgebra, which was denoted by $\widehat{\mf b}$ in \cite{CCC}. We will call $\mf b^r$ the {\em reverse Borel subalgebra}.  
 \begin{exs} \label{ex::ex1} The following quasireductive  Lie superalgebras are of  type I-0:
 	\begin{enumerate}
 		\item[(1)] The reductive Lie algebras $\mf g=\mf g_\oa$. 
 		\item[(2)] The general linear Lie superalgebras $\gl(m|n)$; see Section \ref{sect::271}.
 		\item[(3)] The ortho-symplectic Lie superalgebras $\mf{osp}(2|2n)$; see Section \ref{sect::272}.
 		\item[(4)] The periplectic Lie superalgebras $\pn$; see Section \ref{sect::273}. 
 		\item[(5)]  A semisimple extension
 			\begin{align*}
 			&\g:=(\mf s\otimes \Lambda (\xi))\rtimes \mf d
 			\end{align*} of the Takiff superalgebra determined by a simple Lie algebra $\mf s$ as studied in \cite{CCo}. Here $\mf d=\C\partial_{\xi}\oplus \C\xi\partial_\xi$ denotes the derivation superalgebra of the Grassmann superalgrbra $\Lambda(\xi)$ in odd indeterminate $\xi$. Let $\mf b^{\mf s}$ and $\mf h^{\mf s}$ be a Borel subalgebra and a Cartan subalgebra from a triangular decomposition of $\mf s$, respectively. There is a  type-I gradation  $$\g_{-1}= \mf s\otimes \xi,~\mf g_0=\mf s\oplus \C\xi\partial_{\xi},~\mf g_1= \C\partial_\xi,$$ given by the grading element $-\xi\partial_\xi$ lying in the Cartan subalgebra $\mf h = \mf h^{\mf s}\oplus \C\xi\partial_{\xi}$ with a (distinguished)  Borel subalgebra  given by $\mf b = \mf b^{\mf s}\oplus \C\xi\partial_{\xi}\oplus \g_{1}$; see \cite[Section 2.1]{CCo} for more details.
 	\end{enumerate}
 \end{exs}  
 

 \subsubsection{}  For $\g$ a Lie superalgebra of type I, we recall the Kac functor $K(-): \mf g_\oa$-Mod $\rightarrow \mf g$-Mod defined as $$K(M):=U(\mf g)\otimes_{\mf g_{\geq 0}} M,$$ for any $M\in \mf g_\oa$-Mod by letting $\g_1$ act on $M$ trivially. It is proved in \cite[Theorem A]{CC} that the correspondence $$V\mapsto \wdL(V):=K(V)/\rad K(V)$$  gives rise to a bijection between the sets of isomorphism classes of simple $\g_\oa$- and $\g$-modules. Recall that $\mc S$ denotes the parity shift functor. It turns out that, for any simple $\g_\oa$-module $V$, we have (see also \cite[Corollary 4.5]{CM}): $$\wdL(V) \not \cong \mc S \wdL(V)\cong \wdL(\mc S V).$$  When no confusion is possible, we will not make a distinction between both $\wdL(V)$ and $\mc S \wdL(V)$ and just write $\wdL(V)$.

\subsection{BGG category $\mc O$} Let 
\begin{align}
&\g=\mf n^-\oplus \mf h\oplus \mf n^+. \label{eq::trian}
\end{align} be a triangular decomposition as defined in Section \ref{sect::221}. For type-I Lie superalgebras, we assume that \eqref{eq::trian}  comes from   \eqref{pre::tridec}. The BGG  category $\mc O$ associated to the triangular decomposition  in \eqref{eq::trian} is defined as the  full subcategory of $\g$-Mod consisting of all finitely-generated $\mf g$-modules on which $\mf h$ acts semisimply and $\mf b$ acts locally finitely. Thus $\mc O$ is the category of $\g$-modules restricted  to $\mf g_\oa$-modules by $\Res(-)$ in the classical BGG category $\mc O_\oa$ of $\mf g_\oa$-modules as defined in \cite{BGG}. 
We denote by $\widetilde{\mc F}$ and ${\mc F}$ the category of finite-dimensional $\g$-modules and $\mf g_\oa$-modules, respectively.

\subsubsection{} We recall that the category $\mc O$ has a  structure of  highest weight category. As given in \cite[Section 3.2]{CCC}, we define the partial order $\le$ on $\fh^\ast\times\mZ_2$ as the transitive closure of the relations 
\begin{align}
&(\lambda\pm\alpha,i+j) \le(\lambda,i),~\mbox{for $\alpha\in \Phi(\mf n^\mp_j)$ and $i,j\in\mZ_2$},
\end{align} where $\Phi(\mf n^\mp_j)$ denotes the set of all root in $\mf n^\mp_j$. For any one-dimensional even $\h$-module $\C_\la$ with $\la\in \h^\ast$, the corresponding  Verma module over $\g_\oa$ (resp. over $\g$) is defined as $M(\la):=U(\mf g_\oa)\otimes_{\fb_{\oa}}\C_\lambda$ (resp. $\wdM(\la, \ov i):=\mc S^i(U(\mf g)\otimes_{\fb}\C_\lambda)$ with $i=0,1$) by letting $\mf n_\oa^+$ (resp. $\mf n^+$) act on $\C_\la$ trivially. 

For any $\la\in \h^\ast,~i\in \Z_2$, we denote by $L(\la)$ and $\widetilde{L}(\la, i )$ the simple quotients of   $M(\la)$ and $\widetilde{M}(\la, i)$, respectively. Then the sets $\{L(\la)|~\la\in \h^\ast\}$ and $\{\wdL(\la,i)|~\la\in \h^\ast, i\in \Z_2\}$ form isomorphism classes of simple modules in $\mc O_\oa$ and in  $\mc O$, respectively.  

Suppose that $\mf g$ is of type I-0.  For any $\la \in \h^\ast$ and $i=0,1$, we may observe that   $\widetilde{M}(\la, \ov i)    \cong K(\mc S^i M(\la)).$
Also,  there are canonical epimorphisms $$\widetilde{M}(\la, \ov i)\onto {\mc S}^iK(L(\la)),~{\mc S}^iK(L(\la))\onto \widetilde{L}(\la,\ov i ).$$ 

It will not important for us to make a distinction between simple modules $\wdL(\la,\oa)$ and $\wdL(\la,\ob)$, and so we ignore the parities of their highest weight spaces and just write $\wdL(\la)$. Similarly, we will use $\wdM(\la)$ instead of $\wdM(\la,i)$, for any $(\la,i)\in \h^\ast\times \Z_2$. It turns out that $(\mc O, \leq)$ is a highest weight category with standard objects $\wdM(\la)$ by \cite[Theorem 3.1]{CCC}. Also, we denote by $P(\la)$ and $\widetilde{P}(\la)$ (resp. $I(\la)$ and $\widetilde{I}(\la)$) the projective covers (resp. injective envelopes) of $L(\la)$ and $\widetilde{L}(\la)$ in $\mc O_\oa$ and ${\mc O},$ respectively. 


\subsubsection{} \label{sect::242} The   involution $\sigma$ introduced in Section \ref{Sect::222} leads to a natural duality functor $D$ on the category $\mc O$.  To see this, for a given  $M\in \mc O$  we set  $M^\oast$ to be the restricted dual space of $M$. Then $M^\oast$ is a $\g$-submodule of $\Hom_\C(M,\C)$. Then  we equip $M^\oast$ with a new  structure of $\g$-module by letting 
\begin{align}
&xf(v) = (-1)^{\overline x \overline f} f(\sigma(x)v),
\end{align}  for any homogeneous elements $x\in \g$, $f\in M^\oast$ and any $v\in M$. We denote by $DM$ this resulting $\g$-module. This leads to a contravariant duality  $D$ on $\mc O$; see \cite[Section 2.2.4]{CC} for more details. 

  If $\g$ is a basic classical Lie superalgebra, then the involution $\sigma'$ introduced in  Section \ref{Sect::222} induces the contragredient duality of \cite[Section 13.7]{Mu12}. We will denote by $d$ the  duality induced by $\sigma'$. Observe that $d$ is a simple-preserving duality.

Suppose that the Borel subalgebra $\mf b$  in \eqref{deflu} is determined by  $h\in \mf h$. Let  $\widehat{\mf b}$ denote the Borel subalgebra given by the vector $-w_0(h)$. As observed in \cite[Section 3]{CCC}, the functor $D$ intertwines the standard and costandard objects of $\mc O$ with respect to  the two highest weight category structures given by $\mf b$ and $\widehat{\mf b}$. 

We introduce the standard, costandard objects and the effect of $D$ on them for Lie superalgebra $\g$ of type I.  To see this, let $\g$ be a Lie superalgebra of type I-0 with  $\la\in \h^\ast$. We define $\nbob(\la)$ as the maximal submodule of the coinduced module $\Coind_{\mf n_\oa^-+\g_{-1}}^{\mf g}(\C_{\la})$  on which $\mf h$ acts semisimply and locally finitely. Then $\mc O$ admits  the highest weight category structure given by $\mf b$ with standard objects  $\wdM(\la)$ and  costandard objects $\nbob(\la)$; see \cite[Definition 3.2, Theorem 3.1]{CCC}.

Similarly, we define $\wdM_{\mf b^r}(\la):= U(\g)\otimes_{\mf b^r}\C_\la$, and  we let $\nbobr(\la)$ denote the maximal submodule of the coinduced module $\Coind_{\mf n_\oa^-+\g_1}^{\mf g}(\C_{\la})$ on which $\mf h$ acts semisimply and locally finitely. Then  the  $\mc O$ has the highest weight category structure given by $\mf b^r$ with standard objects $\wdM_{\mf b^r}(\la)$ and costandard objects $\nbobr(\la)$.  

  By \cite[Proposition 3.4]{CCC}, we have isomorphisms 
\begin{align}
&D\wdM(\la)\cong \nbobr(\widehat \la),~D\nbob(\la)\cong \wdM_{\mf b^r}(\widehat \la),\text{ and }D\wdL(\la)\cong \widetilde L_{\mf b^r}(\widehat \la), \label{eq::dualVerma}
\end{align}  where $\widetilde L_{\mf b^r}(\widehat \la)$ denotes the simple highest weight module of highest weight $\widehat \la$ with respect to  $\mf b^r$.

\subsection{Freeness of the action of the root vectors on modules} \label{sect::24f}
For any $\alpha \in \Pi_0$, we choose a (non-zero) root vector $f_\alpha \in \mf g^{-\alpha}_\oa$. A $\g$-module $M$ is said to be {\em $\alpha$-finite} (resp. {\em $\alpha$-free}) if the action of  $f_\alpha$ on $M$ is locally finite (resp. injective).  Similarly, for a given subset $\ups \subset \Pi_0$,  the module  $M$ is called $\ups$-finite (resp. $\ups$-free) if $M$ is $\alpha$-finite (resp. $\alpha$-free), for any $\alpha \in \ups$.

The following lemma   taken from \cite[Lemma 4.12]{CCC} will be frequently used.  
\begin{lem} \label{lem::222}
	Let $\la\in \h^\ast$ and $\alpha\in \Pi_0$. Then either   $\wdL(\la)$ is $\alpha$-finite or else $\wdL(\la)$  is $\alpha$-free. 
\end{lem}

The following useful lemma is a special case of  \cite[Lemma 2.1]{CoM1} (see also \cite[Lemma 4.12]{CCC}):
\begin{lem}  \label{lem::3}
  Suppose that $\g$ is of type I-0. 
	Let $\la\in \h^\ast$ and $\alpha \in \Pi_0$. The following are equivalent:
	\begin{itemize}
		\item[(1)] $\wdL(\la)$ is $\alpha$-finite. 
		\item[(2)]  $D\wdL(\la)$ is $\widehat\alpha$-finite.
		\item[(3)]	$\langle \la+\rho_\oa, \alpha^\vee  \rangle \in \Z_{> 0}$.
	\end{itemize}
\end{lem} 
\begin{proof}
	It follows from \cite[Lemma 2.1]{CM} 
	that Parts (1) and (3) are equivalent. It remains to show that Parts (2) and (3) are equivalent. To see this, we note that by
	\eqref{eq::dualVerma} $D\wdL(\la)\cong \widetilde L_{\mf b^r}(\widehat \la)$, which is   the simple top of $U(\g)\otimes_{\mf g_{\leq 0}} L(\widehat \la)$. Consequently, we have 
	\begin{align}
	&D\wdL(\la):~\widehat \alpha\text{-finite}  \Leftrightarrow L(\widehat \la):~\widehat \alpha\text{-finite} \Leftrightarrow \langle \la+\rho_\oa, \alpha^\vee  \rangle \in \Z_{> 0}.
	\end{align} The conclusion follows.
\end{proof}

\section{Realizations of completion and twisting functors}
 \label{Sect::Realization}
  In \cite{KM2}, Khomenko and Mazorchuk developed the  realizations of Joseph's completion  and  Arkhipov's twisting functors for Lie algebras in terms of the partial approximation   and partial coapproximation functors. Such realizations were also obtained in \cite[Proposition 5.9]{Co16} for basic classical Lie superalgebras. In this section, we will generalize these results to   Lie superalgebras of type I-0.

  Unless mentioned otherwise, we will assume  in this section that $\g$ is either a basic classical Lie superalgebra in \eqref{Kaclist1} or a Lie superalgebra of type I-0.     Recall that we denote by $\mc O^{\Z}$ the full subcategory of $\mc O$ consisting of modules of integral weights. Similarly, we define $\mc O^\Z_\oa\subset \mc O_\oa$.  We will focus on $\mc O^\Z$ and $\mc O^\Z_\oa$ in this section. 
\subsection{Harish-Chandra $(\mf g, \g_\oa)$-bimodules} \label{Sect::HCmod} In this subsection, we consider $\g$ an arbitrary quasireductive Lie superalgebra. 
For a given  $(\widetilde{U},U)$-bimodule $Y$, we denote by $Y^\ad$  the restriction of $Y$ to the adjoint action of $\g_\oa$ given by $x.y =xy-yx$ for any $x\in \g_\oa$ and $y\in Y$.  
Let $\mc B$ denote the category of  finitely-generated $(\widetilde{U}, U)$-bimodules $Y$ for which $Y^{\ad}$ is a  direct sum of modules in $\mc F$.  

For any $\la \in \h^\ast$, we denote by $\chi_\la$ (resp. $\chi^\oa_\la$) the  central character of $\mf g$ (resp, $\g_\oa$) associated with $\la$.  Let $I_\la$ denote the kernel of $\chi_\la^\oa$. Then we have  $\Ann_{U(\mf g_\oa)}M(\la)=U(\mf g_\oa)I_\la$; see,  e.g.~\cite[Theorem~10.6]{Hu08}.
Let $\mc B_\la$ denote the full subcategory of~$\mc B$ consisting of bimodules $Y$ such that~$YI_\la=0$.

We have the left exact functor 
$$\mc L(-,-): U\mbox{-Mod}\times \widetilde{U} \mbox{-Mod}\;\to\;\widetilde U\otimes_{\mC}  U\mbox{-Mod},$$ where $\mc L(M,N)$ is the   maximal $(\widetilde U, U)$-submodule of $\Hom_\C(M,N)$ that  is a  direct sum of finite-dimensional modules with respect to the adjoint action of $\g_\oa$ given by $x\cdot f=xf-fx$, for any $x\in \g_\oa$ and $f\in \Hom_\C(M,N)$. The following lemma is a special case of \cite[Corollary 3.2]{CC}, where the cases of arbitrary dominant regular weights were  considered.
\begin{lem}\label{CorEqiv2}  Let $\g$ be an arbitrary quasireductive Lie superalgebra.  For any dominant, regular and integral  weight $\nu$, we have mutually inverse equivalences $-\otimes_{U}M(\nu)$ and $\cL(M(\nu),-)$ between 
	$\mc B_\nu$ and $\cO^{\Z}.$ 
\end{lem}

In Section \ref{Sect::PfOfThmA}, we will develop an equivalence between $\mc B_\nu$ and a  subcategory $\mc O^{\vpre}$ of $\mc O$ consisting of certain projectively presentable modules, for any dominant (possibly singular)  weight $\nu\in \h^\ast$.

\subsection{The   Joseph's Enright completion functor and Arkhipov's twisting  functor}  \label{sect::32}

Let $\alpha\in \Pi_0$ and $s:=s_\alpha$. We consider     the completion functor $G_s(-): \mc O^\Z\rightarrow \mc O^\Z$ from \cite[Section 4.2]{CC}, which is an analogue of Joseph's version of Enright completion functor $G_s^\oa$ on $\mc O_\oa^\Z$ as introduced in \cite[Section 2]{Jo82}.  The functor  $G_s(-)$ is defined as the following composition:
\begin{align}
&G_s(-):=\mc L(M(s\cdot 0), -)\otimes_{U(\g_\oa)} M(0): \mc O^\Z\ \rightarrow \mc B_0 \xrightarrow{\sim} \mc O^\Z.
\end{align} 
By definition, we have that $\Res \circ G_s=G^\oa_s\circ \Res$ and both $G_s$ and $G_s^\oa$ are left exact functors. 

We denote by $T_s$ (resp. $T_s^\oa$) the Arkhipov's twisting functor on $\mc O^\Z$ (resp. $\mc O_\oa^\Z$) introduced in \cite[Section 5]{CoM1}; see also \cite{AS} and \cite{CMW}. By definition, both $T_s$ and $T_s^\oa$ are right exact. The following theorem is taken from \cite{CC}. 
 
	\begin{thm}\emph{(}\cite[Theorem 4.5]{CC}\emph{)} \label{thm::CC45}  For $\g$ a Lie superalgebra of type I-0 with  $\alpha\in\Pi_0$, we have $D\circ G_{s_{\widehat \alpha}}\circ D \cong T_{s_\alpha}$ on $\mc O^\Z$ and $G_{s_\alpha}$ is right adjoint to~$T_{s_\alpha}$.  \end{thm}
 
For a basic classical Lie superalgebra $\g$  with $\alpha \in \Pi_0$, there  is a version of Joseph's Enright completion functor $G'_{s_\alpha}$ defined in \cite[Section 5]{Co16}.  It is proved in \cite[Theorem 5.5]{Co16} that 
$d\circ G'_{s_\alpha}\circ d \cong T_{s_\alpha}$ and $G'_{s_\alpha}$ is right adjoint to $T_{s_\alpha}$, but we will not use it in the sequel. 

By \cite[Lemma 5.1]{CoM1} and Theorem \ref{thm::CC45}, we have 
\begin{align}
&\Ind \circ T_s^\oa=T_s\circ \Ind,~\Res \circ T_s=T^\oa_s\circ \Res  \text{ and }~\Ind \circ G_s^\oa=G_s\circ \Ind.
\end{align}  
We denote the left derived functor of $T_s$ and the right derived functor of $G_s$ by $\mc LT_s$ and $\mc RG_s$, respectively. By \cite[Proposition 5.11]{CoM1} and arguments therein, we may conclude that the left derived functor $\mc LT_s$ is an auto-equivalence of the bounded derived category $\mc D^b(\mc O^\Z)$ with $\mc RG_s$ as its inverse. 

For a category $\mc C$, we denote by $\Id_{\mc C}$ the identity functor on $\mc C.$ By \cite[Lemma 2.4]{Jo82}, there is a natural transformation $\eta^\alpha_\oa$  from $\Id_{\mc O^{\Z}_\oa}$ to $G_{s}^\oa$ with kernel given by the functor of taking the largest $\alpha$-finite submodule. The following lemma shows that $\eta_\oa^\alpha$ can be lifted to a natural transformation from $\Id_{\mc O^{\Z} }$ to $G_s$.
 
\begin{lem} \label{lem::4} For $\g$ a Lie superalgebra of type I-0, 
 there is a natural transformation $\eta^\alpha: {\emph\Id}_{\mc O^{\Z} } \rightarrow G_{s_\alpha}$ such that $\Res\circ \eta^\alpha=\eta_\oa^\alpha\circ \Res$. In particular, the  $D\eta^\alpha_{D(-)}$ induces a natural transformation from $T_{s_{\widehat \alpha}}$ to ${\emph\Id}_{\mc O^{\Z} }.$
\end{lem}
\begin{proof}	
 The natural embedding $M(s\cdot 0) \hookrightarrow M(0)$ gives rise to a natural transformation $\mc L(M(0), -)\rightarrow \mc L(M(s\cdot 0), -)$. Consequently, we obtain a natural transformation $\eta^\alpha: \Id_{\mc O^{\Z} } \rightarrow G_{s}$ by Lemma \ref{CorEqiv2}. Since we  have  $\Res \circ G_s=G^\oa_s\circ \Res$,  the first conclusion follows from the construction of $\eta_\oa^\alpha$ in \cite[Section 2.4]{Jo82}. The second conclusion then follows from Theorem \ref{thm::CC45}.
\end{proof}

Since   twisting functors satisfy the braid relations, see, for example \cite{CoM1,KM2}, it follows that for any $w\in W$ with a reduced expression $w=s_{\alpha_1}s_{\alpha_2}\cdots s_{\alpha_k}$ ($\alpha_1,\ldots , \alpha_k\in \Pi_0$) the corresponding twisting functor  $T_{w}:=T_{s_{\alpha_1}}\circ T_{s_{\alpha_2}}\circ \cdots \circ T_{s_{\alpha_k}}$ is well-defined.  We use $T^\oa_w$  to denote the corresponding twisting functor  on $\mc O_\oa^\Z$. Then we have  $\Res \circ T_w=T^\oa_w\circ \Res$. Similarly, we define the completion functors $G_w$ and $G_w^\oa$, for any $w\in W$.

\subsection{Partial (co)approximations} 
In this subsection, we shall recall the partial (co)-approximations introduced in \cite{KM2}. 
We fix a subset $\ups \subseteq \Pi_0$ and let $\widehat \ups:=\{\widehat\alpha|~ \alpha\in \ups\}$. The  parabolic   and   Levi subalgebra determined by $\ups$ are denoted by $\fp_\ups$ and $\mf l_\ups$, respectively. We denote  by $w_\ups$ the longest element in the Weyl group $W_{\mf l_\ups}$ of $\mf l_\ups$. 



\subsubsection{}

For any $M\in \mc O$, we let $\mf z_{\ups}(M)$ denote the maximal non-$\ups$ submodule of $M$, namely, $\mf z_{\ups}(M)$ is the maximal $\g$-submodule of $M$ that has no composition factors isomorphic to $\ups$-free simple module. 
Similarly, we the functor  $\mf z_\ups^\oa$  on $\mc O_\oa.$ 
The following lemma shows that $\Res(-)$ intertwines these functors. 

\begin{lem}\label{lem::max1}  
 We have  $$\Res \mf z_\ups(M) = \mf z^\oa_\ups \Res(M),$$ for any $M\in \mc O$. 
\end{lem}  
\begin{proof}
  Let $N:=\mf z^\oa_\ups \Res M$. We first show that $N$ is a $\g$-submodule of $M$. To see this, we let $\la \in \h^\ast$ be such that $[N:L(\la)]>0.$ Then there exists $\alpha \in \ups$ such that $L(\la)$ is $\alpha$-finite by our assumption. Note that all composition factors of  the $\g_\oa$-module $\Res \Ind L(\la)\cong U(\mf g_\ob)\otimes L(\la)$ are all $\alpha$-finite. Now, it follows from $$[\Res \Ind N: L(\mu)]= \sum_{\la \in \h^\ast} [N: L(\la)]\cdot[\Res \Ind L(\la): L(\mu)]$$ that every composition factor of $\Res \Ind N$ is not $\ups$-free. Since the $\g_\oa$-submodule $U(\g_\ob)\cdot N$ is an epimorphic image of $\Res\Ind N$, it is contained in $N$, namely, $U(\g_\ob)\cdot N=N.$ Therefore, $N$ is a $\g$-submodule of  $\mf z_\ups(M)$ by Lemma \ref{lem::222}.  
  
  If $[\mf z_\ups(M): \wdL(\la)]>0$,  then there is  $\alpha\in \ups$ such that $\la$ is $\alpha$-finite.  Therefore, every composition factor of $\Res \mf z_\ups(M)$ is not $\ups$-free. Consequently, we have $\Res \mf z_\ups(M)\subseteq N.$ The conclusion follows.
\end{proof}



We define the dual version $\widehat{\mf z}_\ups$ (resp. $\widehat{\mf z}_\ups^\oa$) of $\mf z_\ups$ (resp. $\mf z_\ups^\oa$) as the functor of taking the maximal quotient that has no composition factors isomorphic to $\ups$-free simple module.  The following lemma is a direct consequence of Lemma \ref{lem::max1} and the application of the duality functor $D$.
\begin{lem} \label{lem::1}  Let $\g$ be a Lie superalgebra of type I-0.  Take $M\in \mc O$, we have 
	\begin{align}
	&\Res \widehat{\mf z}_\ups(M) = \widehat{\mf z}_\ups^\oa \Res(M). \label{eq::4}
	\end{align}
\end{lem}


\subsubsection{}
Following \cite{KM2}, we define in this section the partial (co)approximation functors with respect to $\ups$. 

Set $\mf b_\ups: \mc O^\Z \rightarrow \mc O^\Z$ to be the functor of taking the quotient modulo the maximal submodule, which does not contains $\ups$-free composition factors. The {\em $\ups$-partial approximation functor} $\mf d_\ups$ is defined as the composition of the maximal coextension with non-$\ups$ composition factors, followed by $\mf b_\ups$. Namely, for a module $M$ with its injective hull $I_M$, the module $\mf d_\ups(M)$ is defined to be the intersection of the kernels of all homomorphisms  from $I_M$ to injective envelopes $\widetilde I(\mu)$  with $\ups$-free $\wdL(\mu)$ which annihilates $M$, followed by the functor $\mf b_\ups$.  The dual construction gives the {\em $\ups$-partial coapproximation functor} $\widetilde{\mf d}_\ups$. We refer to \cite[Section 2.5]{KM2} for more details of these functors, including their definitions on  morphisms; see also \cite[Section 5.5]{Co16}. 

 By definition we have natural transformations $\mf d_\ups^{nat}: \Id_{\mc O^\Z}\rightarrow \mf d_\ups$ and $\widetilde{\mf d}_\ups^{nat}: \widetilde{\mf d}_\ups \rightarrow \Id_{\mc O^\Z}$ such that   the kernel of ${\mf d}^{nat}_\ups$ and the cokernel of $\widetilde{\mf d}^{nat}_\ups$ are isomorphic to $\mf z_\ups$ and $\widehat{\mf z}_\ups$, respectively; see \cite[Section 2.5]{KM2}.  
 

\subsection{Realizations} We will realize the twisting and completion functors as partial coapproximation and partial approximation functors on $\mc O^{\Z}$, respectively. 

 \subsubsection{Realizations of $T_s$ and $G_s$}

In this subsection, we fix a simple root $\alpha\in \Pi_0$. We denote $s:=s_\alpha$ and $\mf z_\alpha:= \mf z_{\{\alpha\}}$. Similarly, we define $\mf z_\alpha^\oa$, $\widehat{\mf z}_\alpha$ and $\widehat{\mf z}^\oa_\alpha$, which are known as the (dual) Zuckerman functors associated to $\alpha$.   

 \begin{lem} \label{lem::10}
	Suppose that $\g$ is a Lie superalgebra of type I-0. 
 For any injective module $I\in \mc O^\Z$ and any projective module $P \in \mc O^\Z$, we have the following short exact sequences
 \begin{align}
 &0\rightarrow \mf z_{\alpha}(I)  \rightarrow I \xrightarrow{\eta^\alpha_I}    G_{s}(I)\rightarrow 0,\label{eq::1}\\
 &0\rightarrow  T_{s}(P)\xrightarrow{D\eta^{\widehat  \alpha}_{DP}}   P \rightarrow    \widehat{\mf z}_{\alpha}(P) \rightarrow 0. \label{eq::2}
 \end{align}
 \end{lem}
\begin{proof} Since $\Res(-)$ is left adjoint to $\Coind(-)$, we know that $\Res I$ is injective in $\mc O_\oa$. By \cite[Lemma 2.4]{Jo82} and Lemma \ref{lem::4} we have the following  short exact sequence 
	$$0\rightarrow    {\mf z}^\oa_{\alpha}\Res I  \rightarrow \Res I \xrightarrow{\Res \eta^\alpha_{I}}    G_{s}^\oa(\Res I)\rightarrow 0.$$ Therefore, the kernel of the map $\eta^\alpha_I$ from $I$ to $G_{s}(I)$ is exactly the maximal $\alpha$-finite $\g_\oa$-submodule of $\Res I$, 
	   which turns out to be $\mf z_\alpha(I)$  by the proof of Lemma \ref{lem::max1}. This proves  the exactness of the sequence in \eqref{eq::1}.
  
  Applying the duality $D$ to the short exact sequence \eqref{eq::1} for $I:=DP$, we have the following short exact sequence 
  \begin{align*}
  &0\rightarrow  DG_{s}D P\xrightarrow{D\eta^{\alpha}_{DP}}   P \rightarrow   D   {\mf z}_{\alpha}DP \rightarrow 0. 
  \end{align*} Replacing  $\alpha$ with $\widehat{\alpha}$, the conclusion follows from Theorem \ref{thm::CC45} and Lemma \ref{lem::3}.
\end{proof}


The following theorem is an analogue  of \cite[Proposition 5.9]{Co16} and \cite[Theorem 8, Theorem 10]{KM2}, where the cases of basic classical Lie superalgebras and Lie algebras were considered. 
\begin{thm}  \label{thm::11}
  Suppose that $\g$ is either a basic classical Lie superalgebra in \eqref{Kaclist1} or a Lie superalgebra of type I-0.
	Then we  have $T_{s} \cong \widetilde{\mf d}_{\{\alpha\}}$ and $G_{s}\cong \mf d_{\{\alpha\}}$ as functors on $\mc O^{\Z}$. 
\end{thm}
\begin{proof}  
	The conclusion for the case of basic classical Lie superalgebra has been established in  \cite[Proposition 5.9]{Co16}. Therefore, we focus on the case of Lie superalgebra $\g$ of type I-0.
	
	To complete the proof of this conclusion, we will use the Comparison Lemma in \cite[Lemma 1]{KM2} and its dual version. 	We first show $T_{s} \cong \widetilde{\mf d}_{\{\alpha\}}$, and we shall proceed with an argument similar to the proof of \cite[Theorem 10]{KM2}. Note that  both functors are right exact and we have the following  natural transformations $$T_{s} \xrightarrow{D\eta^{\widehat \alpha}_{D(-)}} \Id_{\mc O^\Z},~\widetilde{\mf d}_{\{\alpha\}} \xrightarrow{\widetilde{\mf d}_{\{\alpha\}}^{nat}} \Id_{\mc O^\Z}.$$


By Lemma \ref{lem::10}, for any projective module $P \in \mc O$ we have the following commutative diagram with exact rows
$$\xymatrix{
0 \ar[r]  & T_{s}(P)\ar[r]^{D\eta^{\widehat \alpha}_{DP}} &  P \ar[r] \ar[d]^{\Id_P}  &   \widehat{\mf z}_{\alpha}P \ar[r] \ar[d]^{\cong} &0,\\
0\ar[r] &  \widetilde{\mf d}_{\{\alpha\}}(P) \ar[r]^{ \widetilde{\mf d}_{\{\alpha\}}^{nat}}  &  P \ar[r] &    \widehat{\mf z}_{\alpha}P \ar[r] &0.}$$
 which implies that $T_{s}(P)\cong \widetilde{\mf d}_{\{\alpha\}}(P)$ since $ \widehat{\mf z}_{\alpha}(P)$ is the unique maximal $\alpha$-finite $\mf g$-quotient of $P$. Now the statements follows from the dual version of Comparison Lemma \cite[Lemma 1]{KM2}.
 
The isomorphism $G_{s}\cong \mf d_{\{\alpha\}}$ can proved similarly using the Comparison Lemma \cite[Lemma 1]{KM2}. Alternatively, it follows from Theorem \ref{thm::CC45} that $G_{s_{\alpha}}\cong D\circ T_{s_{\widehat \alpha}}\circ D\cong D\circ\widetilde{\mf d}_{\{\widehat \alpha\}} \circ D \cong \widetilde{\mf d}_{\{\alpha\}}$. This completes the proof. 
\end{proof}

 \section{The category of projectively presentable modules} \label{Sect::35}
  Throughout this section, we  fix   a dominant weight   $\nu\in \h^\ast$. Also, we set $$\ups=\ups_\nu:=\{\alpha \in \Pi_0|~\langle \nu +\rho_\oa, \alpha^\vee \rangle =0\}.$$
 Let $\mf l_\nu \equiv \mf l_{\ups}$ be the Levi subalgebra of $\g_\oa$ induced by $\ups$, namely, $\mf l_\nu$ is generated by $\g_\oa^\alpha$ ($\pm \alpha\in \ups$) and $\h$. Denote by $W_\nu\equiv W_{\mf l_{\ups}}$ the Weyl group of $\mf l_\nu$.
 
 As formulated in Theorem A of Section \ref{sec1}, the category $\mc O^{\vpre}$ plays an important role in the annihilator-preserving equivalence, and the simple objects $S_\nu(\mu)$ play the role of simple inputs in the construction of simple Whittaker modules. In this section we study, in detail, the modules $S_\nu(\mu)$ and the structure of $\mc O^{\vpre}$. In particular, we will realize $S_\nu(\mu)$ as images of simple highest weight modules under the twisting functors.
  
  \subsection{The category $\mc O^{\vpre}$}   \label{sect::411}
  We recall the set $\mc  X(\nu)$ from Section \ref{sec1} and define a related set $\mc  X_0(\nu)$  as follows: 
  \begin{align}
  &\mc X(\nu):=\{\mu\in \nu+ \mc X|~\wdL(\mu)  \text{ is $\ups$-free} \},\\
  &\mc X_0(\nu):=\{\mu\in \nu+ \mc X|~L(\mu)  \text{ is $\ups$-free} \}.
  \end{align} 
  Since $\langle \nu+\rho_\oa, \alpha^\vee \rangle=0$ for $\alpha \in \ups$, we know that  $$\mc X_0(\nu)= \{\mu\in \nu+ \mc X|~\mu \text{ is $W_\nu$-anti-dominant}\}.$$
  We have $\mc X(\nu)\supseteq \mc X_0(\nu)$; see, e.g., \cite[Lemma 2.1]{CoM1}. Furthermore, assume that $\mf l = \mf l_\nu$ is a Levi subalgebra of $\g$ from \eqref{deflu}  in the way that the corresponding parabolic subalgebra $\mf l+\mf u^+$ contains $\mf b$. Then, every $\alpha\in \ups$ is still a simple root in $\mf n^+$ and $\g^\alpha_\ob=0$, hence, we have  $\mc X(\nu) =\mc X_0(\nu)$  by \cite[Lemma 2.1]{CoM1}. 
  
  If $\g$ is  of type I and $\nu\in \mc X$  then we also have  $\mc X(\nu) =\mc X_0(\nu)$; see also \cite[Lemma 4.13]{CC}. But, in general we have $\mc X(\nu)\neq \mc X_0(\nu)$, see, e.g., examples in \cite[Section 4.3]{CCC}. Observe that if $\nu$ is regular and integral then $\mc X=\mc X_0(\nu)$.

  Let $\mc P_\nu$ (resp. $\mc I_\nu$) denote the  full subcategory of $\mc O$ consisting of modules isomorphic to direct sums of  projective covers $\widetilde  P(\mu)$ (resp. injective modules $\widetilde{I}(\mu)$)  with $\mu \in \mc X(\nu).$  A module $M\in \mc O$ is said to be  {\em projectively presentable} by $\mc P_\nu$ (or {\em $\mc P_\nu$-presentable}) if $M$ has a presentation $$P_1 \rightarrow P_2\rightarrow M\rightarrow 0,$$ with $P_1, P_2\in \mc P_\nu$.  We define      $\mc O^{\vpre}$ as the full subcategory of $\mc O$ consisting of  $\mc P_\nu$-presentable modules.   Similarly, we define $\mc O^{\cvpre}$  to be the full subcategory of $\mc O$ consisting of {\em injectively co-presentable} (or {\em $\mc I_\nu$-co-presentable)} modules by injective modules in $\mc I_\nu$.
   
     Also,  we define $\mc O^{\vpre}_\oa$ (resp.  $\mc O^{\cvpre}_\oa$) as the full subcategory of modules in $\mc O_\oa$ which are presentable (resp. co-presentable) by projective  modules $P(\mu)$ (resp. injective modules $I(\mu)$) with $\mu\in \mc X_0(\nu)$; see also \cite[Section 2]{MS2}.

   If $\g$ is a basic classical Lie superalgebra, then using the contragredient duality $d$ from Section \ref{sect::242} we obtain an equivalence of categories 
  \begin{align}
  &d(-): \mc O^{\cvpre} \xrightarrow{\cong} \mc O^{ \nu\text{-pres}}.\label{eq::Dpresbasic}
\end{align}    If $\g$ is a Lie superalgebra of type I-0. Then it follows from  Lemma \ref{lem::3} that 
  \begin{align}
  &D(-): \mc O^{\cvpre} \xrightarrow{\cong} \mc O^{\widehat \nu\text{-pres}}.\label{eq::Dpres}
  \end{align}

 The way these categories $\mc O^{\vpre}_\oa$, $\mc O^{\cvpre}_\oa$ come into play in the representation theory for   Lie algebras originates in the realization of (possibly singular) categories of Harish-Chandra bimodules as developed in \cite{BG, MS}. Moreover, they are equivalent to certain module category of Enright-complete modules; see \cite{KoM}. The category $\mc O^{\vpre}_\oa$ has  been further studied in more detail in \cite{MaSt04}. 
 
 We will establish in Lemma \ref{thm::14} an equivalence
 $\mc O^{\vpre}\cong \mc B_\nu$. It turns out that $\mc O^{\vpre}$ is equivalent to a   category of Whittaker modules as mentioned in Section \ref{sect::16}; see Theorem \ref{thm::15}.

 The goal of this section is to investigate the category $\mc O^{\vpre}$ and to realize simple objects in $\mc O^{\vpre}$ as twisted simple modules.
  
 \subsection{Connection with $\mc O^{\vpre}_\oa$} \label{Sect::351} 


 For any module $M\in \mc O$, we let $\text{Tr}_{\ups_\nu}(M)$ denote the  sum of all images of homomorphisms from modules in $\mc P_\nu$ to $M$.  For any $\mu \in \mc X(\nu)$, let $A(\mu)$ denote the kernel of the canonical epimorphism $\widetilde P(\mu) \onto \widetilde M(\mu)$. We define \begin{align}
&S_{\nu}(\mu):= \widetilde P(\mu)/\text{Tr}_{\ups_\nu}(\rad \widetilde P(\mu)),\\
&\Delta_{\nu}(\mu):= \widetilde P(\mu)/\text{Tr}_{\ups_\nu}(A(\mu)).
\end{align} 

 The set $\{S_\nu(\mu)|~\mu \in \mc X(\nu)\}$ is an  exhaustive list  of simple objects  in $\mc O^{\vpre}$. 
 For $\g=\g_\oa$ a Lie algebra, the modules $\Delta_\nu(\mu)$ play the role of  proper standard objects in the principal blocks of $\mc O^{\vpre}_\oa$; see \cite[Section 2.6]{MaSt04}. For each $\mu \in \mc X(\nu)$,  $\widetilde P(\mu)$ is the projective cover of $S_\nu(\mu)$ in $\mc O^{\vpre}$.

 The following theorem gives a description of $\mc O^{\vpre}$. In particular,  $\mc O^{\vpre}$ is a   {\em Frobenius extension} of $\mc O^{\vpre}_\oa$ in the sense of \cite[Definition 2.1]{Co}.

 \begin{thm}\label{thm::thm9}  
 	We have the following facts:
 	\begin{itemize}
 		\item[(1)] The functors $$\Ind(-): \mc O^{\vpre}_\oa\rightarrow \mc O^{\vpre},~\Res(-):\mc O^{\vpre}\rightarrow \mc O^{\vpre}_\oa,$$ are well-defined.  In particular, $\Ind P(\mu)$ is a direct sum of projective modules $\widetilde P(\gamma)$ with $\gamma\in \mc X(\nu)$ provided that $\mu\in \mc X_0(\nu)$, as well as  $\Res \widetilde{P}(\mu)$ is a direct sum of projective modules $P(\gamma)$ with $\gamma\in \mc X_0(\nu)$ provided that $\mu\in \mc X(\nu)$. 
 		\item[(2)] Let $M\in \mc O^\Z$.
 		Then we have \begin{align}
 		&\Res M \in \mc O^{\cvpre}_\oa \Rightarrow  M \in \mc O^{\cvpre}. \label{eq::312}
 		\end{align} 
 	\end{itemize}
 \end{thm}
\begin{proof}  Let $\mu,\gamma\in \h^\ast$ be   such that $\mu \in \mc  X_0(\nu)$. Suppose that $\widetilde P(\gamma)$ is a direct summand of $\Ind P(\mu)$. Then it follows from 
	\begin{align} 
	&\Hom_\mf g(\Ind P(\mu), \wdL(\gamma)) = \Hom_{\g_\oa}(P(\mu), \Res \wdL(\gamma)) = [\Res \widetilde  L(\gamma): L(\mu)], \label{eq::310}
	\end{align}  that $[\Res \widetilde{L}(\gamma): L(\mu)]\neq 0$.  We may conclude that $\gamma\in \mc X(\nu)$.   
	
	Let $N\in \mc O^{\vpre}_\oa$. Then there is an exact sequence 
	$P_1\rightarrow P_2\rightarrow N\rightarrow 0$ for some projective modules $P_1,P_2\in \mc O_\oa$ such that every simple summand of $\text{top}P_1$ and $ \text{top}P_2$ has highest weight lying in $\mc X_0(\nu)$. Applying the functor $\Ind(-)$ then we obtain the exact sequence $\Ind P_1\rightarrow \Ind P_2\rightarrow \Ind N\rightarrow 0$.    Since $\Ind P_1, \Ind P_2 \in \mc P_\nu,$ it follows that  $\Ind N \in \mc O^{\vpre}$. 
	
	Conversely, let $M\in \mc O^{\vpre}$. Then we have a short exact sequence
	\begin{align}
	&Q_1\rightarrow Q_2\rightarrow M \rightarrow 0, \label{eq::38}
	\end{align} for some 
	 $Q_1, Q_2\in \mc P_\nu$. Let $\mu \in \mc X(\nu)$, we claim that $\Res \widetilde P(\mu)$ is a  direct sum of projective modules $P(\gamma)$ with $\gamma \in \mc X_0(\nu).$ To see this,  let $P(\gamma)$  be a direct summand of $\Res \widetilde P(\mu)$. Then it follows that
		\begin{align} 
	&0\neq \Hom_{\mf g_\oa}(\Res \widetilde P(\mu), L(\gamma)) = \Hom_{\g}(\widetilde P(\mu), \Coind L(\gamma)) = [  \Coind L(\gamma): \wdL(\mu)]. 
	\end{align} Suppose on the contrary that $\gamma$ is $\alpha$-finite, for some $\alpha\in \ups$. Then $\Coind L(\gamma)$ is $\alpha$-finite as well. But it follows that $\wdL(\mu)$ is $\alpha$-finite, a contradiction.  
	
	Now, applying the restriction functor $\Res(-)$ to \eqref{eq::38} we get a short exact sequence 
	\begin{align}
	&\Res Q_1 \rightarrow \Res Q_2 \rightarrow \Res M \rightarrow 0,
	\end{align} such that any of $\Res Q_1, \Res Q_2$ is a direct sum of modules $P(\gamma)$ with $\gamma\in \mc  X_0(\nu)$. This proves part $(1)$.

Next, we shall prove Part $(2)$. Observe that $\nu \in \mc X$ since $M\in \mc O^\Z$. 
 We will use the results about Enright-complete modules (associated to $\nu$) as developed in \cite{KoM}. 

For $\alpha \in \Pi_0$, recall that we have picked a non-zero root vector  $f_\alpha\in \g_\oa^{-\alpha}$. Let $\text{Loc}_\alpha(-):\g\text{-Mod}\rightarrow \g\text{-Mod}$ denote the  functor of taking locally $\alpha$-finite vectors of modules. We  recall   the Deodhar-Mathieu's version of Enright completion functor $r_\alpha(-):\mc O^\Z\rightarrow \mc O^\Z$   from \cite[Section 2,3]{KoM} as follows:
\begin{align}
&r_\alpha(-):=\text{Loc}_\alpha(\Res^{\widetilde U_\alpha}_{\widetilde U}(\widetilde U_\alpha\otimes_{\widetilde U} -)),
\end{align} where $\widetilde U_\alpha$ denotes the Ore localization with respect to the set $\{f^m_\alpha|~m\geq 0\}$; see, e.g., \cite[Section 3]{KoM} and \cite[Section 2.2]{KM2} for more details. With slight abuse of notation, we denote the corresponding endo-functor on $\mc O_\oa$ by $r_\alpha$ as well. A module $M \in \mc O$ (resp. $M \in \mc O_\oa$) is said to be {\em $\ups$-complete} if $r_\alpha(M) = M,$ for any $\alpha \in \ups.$ 
 
  Let us continue to prove \eqref{eq::312}. By \cite[Theorem 2]{KoM}, we know that $\Res M$ is $\ups$-complete.     By \cite[Lemma 3]{KoM}, it follows that $\text{soc}\Res M$ is $\ups$-free. Therefore, if  $\wdL(\mu)$ is a summand of $\text{soc} M$, then $\mu \in \mc X(\nu)$. This implies that $\text{soc}M$ is $\ups$-free. Now, let $I_M$ be the injective hull of $M$ in $\mc O$. Then $I_M = \bigoplus_{i=1}^\ell I(\la_i)$, for some $\la_i\in \mc X(\nu)$. For any $\alpha \in \ups$, we have  
  \begin{align*}
  &r_\alpha(I/M) \supseteq r_\alpha(I)/r_\alpha(M) \cong I/M, 
  \end{align*} since $r_\alpha$ is right exact. By definition, the socle of $r_\alpha(I/M)$ is $\alpha$-free and so is $\text{soc}(I/M)$. Consequently, there is $I'\in \mc I_\nu$ such that $I/M \hookrightarrow I'$. This proves $M\in \mc O^{\cvpre}$, as desired. This completes the proof. 
\end{proof}

 The following is a generalization of \cite[Corllary 2.11]{MaSt04}, where the cases of Lie algebras were considered. 
\begin{cor} Suppose that $\g$ is a Lie superalgebra of type I-0. Then,  for any $\nu\in \mc X$ and $\la \in \mc X(\nu)$ we have  $\widetilde I(\la)\in \mc O^{\vpre}.$ 
	\end{cor}
\begin{proof} 
	By duality in \eqref{eq::Dpres}, it suffices to show that $\widetilde  P(\la)\in \mc O^{\cvpre}$. By Part (2) in Theorem \ref{thm::thm9}, we only need to show that $\Res \widetilde  P(\la) \in \mc O_\oa^{\cvpre}$. By Part (1) of Theorem \ref{thm::thm9}, we know that $\Res \widetilde P(\la) =\bigoplus_{i=1}^{\ell} P(\la_i),$ for some $\la_i\in \mc X_0(\nu).$  For any $\mu\in \mc X_0(\nu)$, it follows from \cite[Lemma 4, Theorem 1, Theorem 2]{KoM} (or \cite[Section 2.11]{MaSt04}) that $P(\mu)\in \mc O^{\cvpre}_\oa$. In particular, we have $\Res P(\la)\in \mc O^{\cvpre}_\oa$. The conclusion follows. 
\end{proof}

\subsection{Annihilator ideals of $S_\nu(\mu)$} \label{sect::421}
  

 We continue to assume that $\nu\in \h^\ast$ is dominant.  Recall that we denote by  $w_\ups$ the longest element in $W_\nu$. The following theorem is the main result in this subsection that are to be used in the sequel. 
 \begin{thm} \label{thm::realization}
 	 Suppose that $\g$ is either a basic classical Lie superalgebra in \eqref{Kaclist1} or a Lie superalgebra of type I-0.
 	For any $\la\in \mc X$, we have $T_{w_\ups} \wdL(\la)\cong \wfu \wdL(\la).$ 	In particular, if $\nu\in \mc X$ and  $\la \in \mc X(\nu)$ then we have \begin{align}
 	&T_{w_\ups} \wdL(\la)\cong S_\nu(\la), \label{eq::91}\\
 	&T_{w_\ups} \wdM(\la)\cong  \Delta_\nu(\la). \label{eq::92}
 	\end{align}  
 \end{thm}
 \begin{proof}
 	We first assume that $\wdL(\la)$ is $\alpha$-finite for some $\alpha \in \ups$. Then $\widetilde{\mf d}_\ups \wdL(\la) =0$ by definition. Let $w_\ups=s_{\alpha_1}\cdots s_{\alpha_\ell}$ with  $\alpha_\ell = \alpha$ being a reduced expression of $w_\ups$. Then $T_{w_\ups} \wdL(\la) \cong T_{s_{\alpha_1}}\circ \cdots \circ T_{s_{\alpha_\ell}}\wdL(\la)$, which is zero since $T_{s_{\alpha_\ell}}\wdL(\la)=0$  by Theorem \ref{thm::11}, or alternatively by $\Res T_{s_{\alpha_\ell}}\wdL(\la) \cong   T^\oa_{s_{\alpha_\ell}}\Res \wdL(\la)=0.$
 	
 	Next, we assume that $\nu \in \mc X$ and $\la \in \mc X(\nu)$. Observe that we have    $\widetilde{\mf d}_\ups \wdL(\la) = \widetilde P(\la)/\text{Tr}_{\ups}(\rad \widetilde P(\la)) =S_\nu(\la)$ by definition.  Let 
 	\begin{align}
 	&w_\ups=s_{\alpha_1}\cdots s_{\alpha_\ell} \label{eq::rex}
 	\end{align} be a reduced expression of $w_\ups$. Then 
 	\begin{align}
 	&T_{w_\ups}\wdL(\la) \cong T_{s_{\alpha_1}}\circ \cdots \circ T_{s_{\alpha_\ell}}\wdL(\la).
 	\end{align}  
 	Set $T_i = T_{s_{\alpha_i}}$, for $i=1,\ldots,\ell$. We are going to investigate the structure of $T_{w_\ups}\wdL(\la)$ explicitly using Theorem \ref{thm::11}. To see this, we first note that $\widetilde P(\la)$ is the projective cover of $\widetilde P(\la)/\text{Tr}_{\{\alpha_{\ell}\}}(\rad \widetilde P(\la)),$ the latter is $T_{\ell}\wdL(\la)$ since $\wdL(\la)$ is $\ups$-free.   Therefore, we have  
 	$$T_{\ell-1}T_{\ell}\wdL(\la)\cong \widetilde P(\la)/\text{Tr}_{\{\alpha_{\ell-1}\}}(\text{Tr}_{\{\alpha_{\ell}\}}(\rad \widetilde P(\la))).$$ 
 	In general, if we set $M:=\text{Tr}_{\{\alpha_1\}}(\cdots(\text{Tr}_{\{\alpha_{\ell-1}\}}(\text{Tr}_{\{\alpha_{\ell}\}}(\rad \widetilde P(\la))))\cdots)$ then it follows that  
 	\begin{align}
 	&T_{w_\ups}\wdL(\la) = T_1 \cdots T_{\ell-1}T_{\ell}\wdL(\la)\cong\widetilde P(\la)/M
 	\end{align} using the same argument. 
 	
 	To prove that $M= \text{Tr}_\ups(\rad \widetilde P(\la))$, it remains to show the inclusion $M\subseteq \text{Tr}_\ups(\rad \widetilde P(\la)).$ We shall show that the   $\text{top} M = M/\rad M$ is a (direct) sum of $\ups$-free simple modules. Observe that $M$ is an image of a   sum of projective covers  with $\alpha_\ell$-free tops, and so $\text{top} M$ is a direct sum of  $\alpha_\ell$-free simple modules.  Let $\alpha\in \ups$ be arbitrary, there exists a reduced expression \eqref{eq::rex} such that $\alpha_\ell = \alpha$. Consequently, the $\text{top}M$ is $\ups$-free.  Let $\text{top}M$ be an epimorphic image of a projective module $Q$ that has a $\ups$-free top.  Then $M$ is an epimorphic image of $Q$ since $\rad M$ is a  superfluous submodule of $M$ (see, e.g., \cite[Section 2.2]{CCM}), which implies that  $M\subseteq \text{Tr}_\ups(\rad \widetilde P(\la))$. This proves \eqref{eq::91}. 
 	 
  The isomorphism in \eqref{eq::92} can be proved in similar fashion. This completes the proof.

 \end{proof}

 
 \begin{lem} \label{lem::15}    Suppose that $\g$ is either a basic classical Lie superalgebra in \eqref{Kaclist1} or a Lie superalgebra of type I-0. We have $${\emph \Ann}_{U(\g)} S_\nu(\mu) ={\emph \Ann}_{U(\g)} \widetilde{L}(\mu),$$ for any $\nu$ integral and $\mu \in \mc X(\nu).$
 \end{lem}
 \begin{proof}  By Theorem \ref{thm::realization}, it suffices to show that $\Ann_{U(\g)}  T_{w_\ups}\wdL(\mu) =\Ann_{U(\g)} \wdL(\mu)$.  Let $w_\ups = s_{\alpha_k}\cdots s_{\alpha_1}$ with $\alpha_1, \ldots,\alpha_k \in \Pi_0$ be a reduced expression of $w_\ups$ in $W_\nu$.
 	
 	By the right exactness of twisting functors  and  \cite[Theorem 5.12(ii)]{CoM1}, the module $\wdL(\mu)$ is an epimorphic image of  $T_{w_\ups} \wdL(\mu)$. Therefore, we have $\Ann_{U(\g)} T_{w_\ups} \wdL(\mu)\subseteq \Ann_{U(\g)} \wdL(\mu).$

 	For a given $\alpha = \alpha_i$, let $\varphi_\alpha$ be the associated automorphism of $\g$ used in the construction of $T_{s_\alpha}$, mapping  $\g_\ell^\beta$ to $\g_\ell^{s_\alpha(\beta)}$, for any root $\beta$ of $\g$ and $\ell\in \Z_2$, as in  \cite[Section 5]{CoM1} and \cite[Section 2]{CMW}.  Also, it is shown in  \cite[Lemma 5.15]{CoM1} that 	 $\varphi(\Ann_{U(\g)} \wdL(\mu)) = \Ann_{U(\g)} \wdL(\mu)$ and $\Ann_{U(\g)} T_{s_{\alpha_{1}}} \wdL(\mu)= \Ann_{U(\g)} \wdL(\mu)$.

 	Suppose that $\Ann_{U(\g)} T_{s_{\alpha_{i}}\cdots s_{\alpha_1}}\wdL(\mu) = \Ann_{U(\g)}\wdL(\mu)$, for some  $1\leq i \leq k-1$. Using the same argument as given in the proof of  \cite[Lemma 5.15]{CoM1}, we have $$\Ann_{U(\g)} T_{s_{\alpha_{i+1}}\cdots s_{\alpha_1}} \wdL(\mu)\supseteq \varphi(\Ann_{U(\g)} T_{s_{\alpha_{i}}\cdots s_{\alpha_1}}\wdL(\mu)) = \Ann_{U(\g)}\wdL(\mu),$$ and so $\Ann_{U(\g)} T_{s_{\alpha_{i+1}}\cdots s_{\alpha_1}} \wdL(\mu) =  \Ann_{U(\g)}\wdL(\mu)$. Consequently, 
 	we have $\Ann_{U(\g)} T_{w_\ups} \wdL(\mu)= \Ann_{U(\g)} \wdL(\mu).$
 \end{proof}
 
  We will give in Corollary \ref{cor::306} an alternative proof of Lemma \ref{lem::15} using the Harish-Chandra bimodules in $\mc B_\nu$. To emphasize more utilities of the realization of $S_\nu(\mu)$ in terms of twisted simple modules, the author provides in \cite{Ch213} an application to the study of the abstract Kazhdan-Lusztig theory for the periplectic Lie superalgebras in the sense of Cline, Parshall and Scott.

\section{Annihilator ideals of simple Whittaker modules} \label{Sect::AnnSWhi}
In this section, unless specified otherwise, we let $\g$ be an arbitrary   quasireductive Lie superalgebra with a triangular decomposition $\g =\mf n^-\oplus \mf h \oplus \mf n^+.$ Recall that we refer to the simple $\g$-modules that are locally finite over $\mf n^+$ as the simple Whittaker modules.
\subsection{The categories of Whittaker modules} \label{Sect::41} 
 In this subsection, we shall recall the classification of simple Whittaker modules from \cite{Ch21}. For each $\zeta \in \ch \mf n_\oa^+:=(\mf n_\oa^+/ [\mf n_\oa^+,\mf n_\oa^+])^\ast$, we let $\mf l_\zeta$ denote the Levi subalgebra  of $\g_\oa$ determined by the subset  $\Pi_\zeta:=\{\alpha\in\Pi_0|~\zeta(\mf g_\oa^\alpha)\neq 0\}$, where $\g_\oa^\alpha$ denotes the root space of $\g_\oa$ associated to the root $\alpha$.  Denote by $W_\zeta$ the Weyl  group of $\mf l_\zeta$, namely, $W_\zeta$ is generated  $s_\alpha$ with  $\alpha \in \Pi_\zeta$ when regarded as a subgroup of $W$.  
 Set $\mf l_\zeta = \mf n_\zeta^-\oplus \mf h\oplus \mf n^+_\zeta$ to be  the  triangular decomposition of $\mf l_\zeta$. For any $\la \in \mf h^\ast$, the Kostant's simple Whittaker modules are defined as follows (see  \cite{Ko78}):
 \begin{align}
 &Y_\zeta(\la, \zeta):=U(\mf l_\zeta)/\text{Ker}(\chi^{\mf l_\zeta}_\la) U(\mf l_\zeta) \otimes_{U(\mf n^+_\zeta)}\mathbb C_\zeta,
 \end{align} where  $\text{Ker}(\chi^{\mf l_\zeta}_\la)$ is the kernel of the central character $\chi_\la^{\mf l_\zeta}$ of $\mf l_\zeta$ and $\C_\zeta$ is the one-dimensional $\mf n^+_\zeta$-module associated with $\zeta$. 
 Let $\mf q_\zeta$ be the parabolic subalgebra of $\mf g_\oa$ having $\mf l_\zeta$ as its Levi subalgebra  associated to the parabolic decomposition of $\g_\oa$ induced by $\Pi_\zeta$. 
 Following \cite{Mc,MS} (see also \cite{B}), the  standard Whittaker modules over $\g_\oa$ is defined to be $M(\la, \zeta):= U(\g_\oa)\otimes_{\mf q_\zeta} Y_\zeta(\la, \zeta).$  
 
 \subsubsection{} \label{sect::511} We recall the category  $\widetilde{\mc N}$ from \cite[Section 1.5]{Ch21}, namely, $\widetilde{\mc N}$  consists  of finitely-generated $\g$-modules that are locally finite over $\mf n^+$ and over the center $Z(\mf g_\oa)$. 	For any $\zeta \in \ch \mf n_\oa^+$, let $\widetilde{\mc N}(\zeta)$ be the full subcategory of $\widetilde{\mc N}$ consisting of modules $M \in \tN$  such that $x-\zeta(x)$  acts locally nilpotently on $M$, for each $x\in \mf n_\oa^+$.

 To formulate and construct the standard and simple Whittaker modules, we treat the cases of basic classical Lie superalgebras and type-I Lie superalgebras separately; see \cite[Section 3.1, Section 3.2]{Ch21}.
   First, suppose that $\g$ is a Lie superalgebra of type I, then  
 we regard $M(\la, \zeta)$ as a $\g_\oa$-supermodule by assigning a parity and define the corresponding {\em standard Whittaker module} over  $\g$:
\begin{align} 
&\widetilde{M}(\la, \zeta):=   K(M(\la, \zeta)),\label{eq::52}
\end{align}   for any $\la \in \h^\ast.$
 As mentioned in Section \ref{sect::pre1}, we will ignore the parity of $M(\la, \zeta)$ and just write $\wdM(\la,\zeta)$.  
 
 Next,   assume that $\g$ is a basic classical Lie superalgebra in \eqref{Kaclist1}. We suppose that  $\mf l_\zeta$ is the Levi subalgebra in a parabolic decomposition $\mf g=\mf u_\zeta^- \oplus \mf l_\zeta\oplus \mf u_\zeta^+$ of $\g$. For any $\la \in \h^\ast$, the corresponding {\em standard Whittaker module} over  $\g$ is  defined as
 \begin{align}  
 &\widetilde{M}(\la, \zeta):= {U}(\g)\otimes_{\mf p} Y_\zeta(\la, \zeta),\label{eq::53}
 \end{align} where $\mf p:=\mf p_\zeta=\mf l_\zeta \oplus \mf u_\zeta^+$ is the corresponding parabolic subalgebra. 
 
 We may observe that the standard Whittaker modules defined in \eqref{eq::52} and \eqref{eq::53} are isomorphic in the case when $\g$ is a Lie superalgebra of type I and the $\mf l_\zeta$ is a Levi subalgebra such that $\mf p_\zeta$ contains the Borel subalgebra $\mf b$ from \eqref{pre::tridec}. We remark that this is true for arbitrary $\zeta\in \ch \mf n_\oa^+$ provided that  $\mf g$ is one of $\gl(m|n),~\mf{osp}(2|2n)$ and $\pn$.

  It is proved in \cite[Theorem 6 and Theorem 9]{Ch21}  (see also \cite[Theorem 4.1]{CM}) that  each $\widetilde{M}(\la,\zeta)$ has a simple top in both cases above, which we denote by $\widetilde{L}(\la, \zeta)$. The following theorem  gives an explicit construction and classification of simple Whittaker modules:
\begin{thm} \emph{(}\cite[Theorem 6 and Theorem 9]{Ch21}\emph{)} \label{mainthm1typeI} 
	For any $\la,\mu \in \h^\ast$, $\zeta,\zeta'\in \ch\mf n_\oa^+$,  $\widetilde{L}(\la, \zeta) $ and  $\widetilde{L}(\mu, \zeta')$ are isomorphic (up to parity) if and only if $\zeta=\zeta'$ and $\la \in W_\zeta\cdot \mu$. Moreover, we have 
	
	(1).  Suppose that $\mf l_\zeta$ is a Levi subalgebra of $\g$. Then the set $$\{\widetilde{L}(\la, \zeta)|~\la \in \h^\ast\},$$ is a complete list of simple Whittaker modules in $\widetilde{\mc N}(\zeta)$.  
 
 (2). Suppose  that $\g$ is of type I.  Then the set $$\{\widetilde{L}(\la, \zeta)|~\la \in \h^\ast,\zeta \in \ch \mf n_\oa^+\},$$ is a complete list of simple Whittaker modules over $\g$.
\end{thm}


	Let $S$ be a simple Whittaker module over $\g$. By \cite[Lemma 3 and Corollary 4]{Ch21} there is a simple (Whittaker) $\g_\oa$-submodule $V$ of $\Res S$. We recall from Section \ref{sect::166} that $S$ is said to be of integral type if $V$ admits a central character of $\g_\oa$ associated to an integral weight. The following proposition shows that the   integrability of simple Whittaker modules is well-defined. 
	
	\begin{prop} \label{prop::intWhi}
		Let $S$ be a simple Whittaker module. Let  $L(\mu,\zeta)$ and $L(\gamma,\zeta)$ be two composition factors of $\Res S$. Then $\mu\in \mc X$ if and only if $\gamma\in \mc X$.
	\end{prop}
\begin{proof}
 By \cite[Lemma 3]{Ch21}, we know that $\Res S\in {\mc N}$. Therefore, there exists a simple Whittaker $\g_\oa$-module $L(\la,\zeta)$ such that $L(\la,\zeta)\hookrightarrow \Res S$ since every module in $\mc N$ has finite length. Therefore, $S$ is a quotient of $\Ind L(\la,\zeta)$. It follows that both  $L(\mu,\zeta)$ and $L(\gamma,\zeta)$ are composition factors of $\Res \Ind L(\la,\zeta)\cong U(\mf g_\ob)\otimes L(\la,\zeta)$. By \cite[Proposition 2.1]{MS}, $L(\la,\zeta)$ admits the  central character of $\g_\oa$ associated with the weight $\la$. Consequently, \begin{align}
 &\mu \in \mc X \Leftrightarrow \la \in \mc X \Leftrightarrow\gamma \in \mc X. \label{eq::1616} 
 \end{align} The conclusion follows. 
 
 We also provide an alternative proof as follows. Observe that  $S$ is a quotient of $\Ind M(\la,\zeta)$, the later has a filtration with subquotients $M(\la+\eta)$ with certain $\eta \in \mc X$. Again, by \cite[Proposition 2.1]{MS} $M(\la+\eta)$ admits the central character of $\g_\oa$ associated with the weight $\la+\eta$. Consequently, we have  \eqref{eq::1616}.  
\end{proof}

The following corollary is a consequence of Proposition \ref{prop::intWhi}. 
\begin{cor} 
	The simple Whittaker module $\wdL(\la,\zeta)$ defined in Section \ref{sect::511} is of integral type if and only if $\la\in \mc X.$  
\end{cor}

 \subsubsection{Equivalence of categories} \label{sect::412}
 	As observed in \cite[Section 4.1.2]{Ch21}, every module $M\in \widetilde{\mc N}(\zeta)$ decomposes into generalized eigenspaces $M=\bigoplus_{\la\in \h^\ast}M_{\chi_\la^{\mf l_\zeta}}$ according to the action of elements in $ Z(\mf l_\zeta)$. 
	For any $\la\in \h^\ast$, we put 
	\begin{align*}
	&\widetilde{\mc N}(\la+\mc X, \zeta) := \{M\in \widetilde{\mc N}(\zeta)|~M_{\chi^{\mf l_\zeta}} = 0 \text{ unless }\chi^{\mf l_\zeta} = \chi^{\mf l_\zeta}_\mu, \text{ for some }\mu \in \la+\mc X\}.
	\end{align*}    
 

 For $\la \in \h^\ast$, we define \[M^n(\la, \zeta):= U(\g_\oa)\otimes_{\mf q_\zeta} Y_\zeta^n(\la, \zeta),\]  where $Y_\zeta^n(\la, \zeta):=U(\mf l_\zeta)/(\text{Ker}\chi_\la^{\mf l_\zeta})^n U(\mf l_\zeta)\otimes_{U(\mf n^+_\zeta)}\mathbb C_\zeta$ so that $Y_\zeta^1(\la, \zeta) = Y_\zeta(\la, \zeta)$; see  \cite[Section 5]{MS}.

Suppose that $\nu \in \h^\ast$ is dominant such that $W_\zeta =W_\nu$.  Recall the category $\mc B$ of $(\g,\mf g_\oa)$-bimodules from Section \ref{Sect::HCmod}. Also,  recall that $I_\nu$ denotes the ideal $U(\g_\oa)\Ann_{Z(\fg_{\oa})}M(\nu)$.  Set $\mc B_\nu^\infty$ to be the full subcategory of $\mc B$ consisting of objects $$\{Y \in \mc B|~YI_\nu^n = 0, \text{ for }n>>0\}.$$ 

The following theorem is taken from \cite[Theorem 16 and Propisition 33]{Ch21}. 
\begin{thm} \label{thm::Ch}  The functor $A \mapsto \lim\limits_{\leftarrow}A\otimes_{U(\g_\oa)}  M^n(\nu,\zeta)$ gives an equivalence of categories 
	\begin{align}
	&\widetilde T: \mc B_\nu^\infty \rightarrow \widetilde{\mc N}(\nu +\mc X, \zeta). \label{eq::555}
	\end{align} 
	
	Furthermore, assume that $\g$  is of type I-0. Then the   equivalence $\widetilde T$ maps the top of $\mc L(M(\nu), \widetilde{M}(\mu))$ to   $\widetilde{L}(\mu, \zeta)$ with the same (left) annihilator ideal, for any $\mu \in \nu +\mc X$. 
\end{thm}

	Such an equivalence  \eqref{eq::555} is sometimes referred to a Mili{\v{c}}i{\'c}-Soergel type equivalence. The formulation and idea originated in the work of Mili{\v{c}}i{\'c} and Soergel \cite{MS}.   There are some analogues and variations of Mili{\v{c}}i{\'c}-Soergel type 	equivalences studied further, see, e.g., \cite{ CC, KM, MaMe12, Ma2}.

We recall from  \cite[Section 4.1]{Ch21}  an intrinsic definition of $\widetilde{\mc N}(\nu+\mc X, \zeta)^1$ that is introduced in Section \ref{Sect::14}.  Let $\hat S$ denote the completion of the symmetry algebra $S$ over $\mf h$ at the maximal ideal generated by $\mf h$. Recall the ring  $\hat S^W$ from  \cite[Section 4, Section 5]{MS} consisting of  invariants of $\hat S$ under the action of  $W_\zeta$.  We have already observed that modules in $\widetilde{\mc N}(\zeta)$ restrict to $\mf l_\zeta$-modules that are locally finite over $Z(\mf l_\zeta)$. Therefore for any    $M\in \widetilde{\mc N}(\zeta)$ there is a ring homomorphism  
\begin{align} 
&\theta_M: ~\hat S^W\rightarrow \text{End}_{\mf l_\zeta}M, \label{eq55}
\end{align} as constructed in \cite[Section 4]{MS}. 

For any $s\in \hat S^W$ and any central character $\chi^{\mf l_\zeta}$ of $\mf l_\zeta$, the action of $\theta_{M}(s)$ on $M_{\chi^{\mf l_\zeta}}$ is given by the action of an element $\theta_{(-)}(s)$,  depending only on the choices of $s$ and $\chi^{\mf l_\zeta}$, in the completion of $Z(\mf l_\zeta)$ at the kernel of $\chi^{\mf l_\zeta}$; see also \cite[Section 4.1.1]{Ch21}. Then $\widetilde{\mc N}(\nu+\mc X, \zeta)^1$ is defined as the category of modules 
$M\in \widetilde{\mc N}(\nu+\mc X, \zeta)$ such that  $\theta(\mf m)M=0,$ where $\mf m$ is the maximal ideal of $\hat S^{W}$.  The equivalence $\widetilde T$ leads to an equivalence between $\mc B_\nu$ and $\widetilde{\mc N}(\nu+\mc X, \zeta)^1$, namely, it is shown in \cite[Theorem 14]{Ch21} (for the proof of Theorem \ref{thm::Ch}) that we have the following equivalence   \begin{align}
&-\otimes_{U(\g_\oa)} M(\nu,\zeta):\mc B_\nu \cong \widetilde{\mc N}(\nu+\mc X, \zeta)^1, \label{eq::equA}
\end{align} preserving the   (left) annihilator ideals,  with inverse $\mc L(M(\nu,\zeta),-)$. 


\subsection{Proofs of  Theorem A and Theorem B} \label{Sect::PfOfThmA} 
In this subsection, we shall give complete proofs of the Theorems A and B introduced in Section \ref{sec1}. We need some preparatory notations before giving the proof.

For a given full subcategory $\mc C$ of either $\g$-Mod or $\g_\oa$-Mod, we denote by $\add(\mc C)$ the additive closure of $\mc C$, that is, $\add(\mc C)$ consists of all modules isomorphic to direct sums of  summands 
of objects in $\mc C$. Also, we define  $ \widetilde{\mc F}\otimes \mc C$, (resp. $\mc F\otimes \mc C$, $\Ind \mc C$) to be the full subcategory consisting of objects $E\otimes X$ (resp. $E'\otimes X$,  $\Ind X$), with $X\in \mc C$ and  $E\in \widetilde{\mc F}$ (resp. $E'\in \mc F$).    Also, we let~$\mathrm{coker}(\mc C)$ denote the {\em coker-category} consisting of all modules $N$ that have a presentation 
$$N_1\to N_2 \to N\to 0,$$ where 
$N_1, N_2$ are  modules in $\add(\mc C)$; see  \cite{MaSt08}. For instance, we may note that $\mc O^{\vpre}=\coker(\mc P_\nu)$. The following lemma is a   consequence of \cite[Theorem 3.1]{CC}.

\begin{lem} \label{lem::115} Suppose that $\nu\in\mf \h^\ast$ is dominant. Then the functor $(-)\otimes_{U(\g_\oa)} M(\nu)$ gives rise to an equivalence from   $\mc B_\nu$ to $\coker(\widetilde{\mc F}\otimes \Ind M(\nu))$ with inverse $\mc L(M(\nu),-)$.
\end{lem}



The following theorem is the main result in this section. 
 
\begin{thm}\label{thm::15} Let $\zeta\in \ch \mf n_\oa^+$. Suppose that  $\nu\in \h^\ast$ is dominant  such that $W_\nu =W_\zeta$. 	Then we have the following  annihilator-preserving equivalences of categories
	\begin{align}
	&\mc L(M(\nu),-)\otimes_{U(\g_\oa)} M(\nu,\zeta): \mc O^{\vpre}\xrightarrow{\cong }\mc B_\nu \xrightarrow{\cong} \widetilde{\mc N}(\nu+\mc X, \zeta)^1. \label{eq::16eq}
	\end{align} 
	
	Furthermore, assume that $\g$ is of type I-0 with $\mu\in \mc X(\nu)$. Then the equivalence in \eqref{eq::16eq} sends the $S_\nu(\mu)$ to $\wdL(\mu,\zeta)$, namely, we have $S(\mu,\zeta) =\wdL(\mu,\zeta)$.
\end{thm}
 
 The proof of Theorem \ref{thm::15} follows from \eqref{eq::equA}, Theorem \ref{thm::Ch} and the following lemma. 
\begin{lem} \label{thm::14}
 The functor $(-)\otimes_{U(\g_\oa)} M(\nu)$ leads to an annihilator-preserving equivalence $$\mc B_\nu \cong \mc O^{\vpre},$$
 with inverse $\mc L(M(\nu), -).$    
 
 	Furthermore, assume that $\g$ is of type I-0. Then, for any $\mu \in \mc X(\nu)$, the equivalence $(-)\otimes_{U(\g_\oa)} M(\nu)$ sends the top of $\mc L(M(\nu), \widetilde{M}(\mu))$ to $S_\nu(\mu)$, with the same (left) annihilator ideal.
\end{lem}
\begin{proof}   By Lemma \ref{lem::115}, it remains to show that $\coker(\widetilde{\mc F}\otimes \Ind M(\nu)) = \mc O^{\vpre}$. To see this,   we observe  that 
	\begin{align}
	&\add(\widetilde{\mc F}\otimes \Ind M(\nu)) \subseteq \add (\Ind (\mc F\otimes M(\nu))).
	\end{align}
		Also, we note 
		\begin{align}
		&\add(\Ind(\mc F\otimes M(\nu))) \subseteq \add(\Ind(\mc F \otimes \Res \Ind M(\nu))) \subseteq \add(\widetilde{\mc F} \otimes  \Ind M(\nu)).
		\end{align}
		Consequently, we have $\add(\widetilde{\mc F}\otimes \Ind M(\nu)) = \add (\Ind (\mc F\otimes M(\nu))).$
		
	By \cite[Theorem 3.3]{BG}, 
		 $\add(\Ind(\mc F\otimes M(\nu)))$ is the additive closure of the category of projective modules $\Ind P(\mu)$ with $\mu \in \mc X_0(\nu).$  For a given  $\mu\in \mc X_0(\nu)$, the module $\Ind P(\mu)$ is a direct sum of projective modules $\widetilde P(\gamma)$ with $\gamma \in \mc X(\nu)$ by  Part (1) of  Theorem \ref{thm::thm9}. Hence, we have $\coker(\Ind(\mc F\otimes M(\nu)))\subseteq \mc O^{\vpre}$. 
		 
		  On the other hand, if $\mu \in \mc X(\nu)$, then $\Res\widetilde{P}(\mu)$ is isomorphic to $\bigoplus_{i=1}^k P(\gamma_i)$ with $\gamma_1,\ldots,\gamma_k  \in \mc X_0(\nu)$ by  Part (1) of  Theorem \ref{thm::thm9} again. It follows that $\widetilde P(\mu)$ is an epimorphic image of $\Ind( \bigoplus_{i=1}^k P(\gamma_i))$, and so $\widetilde P(\mu)\in\add(\Ind(\mc F\otimes M(\nu)))$.   It follows that $\coker(\Ind(\mc F\otimes M(\nu))) \supseteq \mc O^{\vpre}$.   The first conclusion follows. 		
		 		

	 Finally, assume that $\g$ is of type I-0. Since $\mc L(M(\nu),-):\mc O\rightarrow \mc B_\nu$ is exact by \cite[Lemma 6.2]{BG} (see also Theorem \ref{thm::24} for a proof),  the projective cover
		of the top of  $\mc L(M(\nu), \widetilde{M}(\mu))$ is 
		$\mc L(M(\nu), \widetilde P(\mu))$, for any $\mu \in\mc  X(\nu)$. Therefore, the equivalence $(-)\otimes_{U(\g_\oa)}M(\nu)$ sends the top of $\mc L(M(\nu), \widetilde{M}(\mu))$ to the simple quotient of $ \widetilde P(\mu)$  in $\mc O^{\vpre}$, namely, the $S_\nu(\mu)$. This completes the proof.

\end{proof}

We remark that Lemma \ref{thm::14} is a generalization of results in \cite[Section 5.9]{BG}, where the case of reductive Lie algebra was considered. 

 We are now in a position to prove Theorem A that is introduced in Section 
  \ref{thm::thmA}.  Before giving the proof, we recall some helpful results from \cite{MS}.
  \begin{lem}{\em(}Mili{\v{c}}i{\'c}-Soergel{\em)} \label{lem::19} Let  $\nu\in \h^\ast$ be dominant with $W_\nu=W_\zeta$. Then we have \begin{itemize}
  		\item[(1)] $\mc N(\nu+\mc X, \zeta)^1$ is stable under tensoring with finite-dimensional $\g_\oa$-modules.
  		\item[(2)] $M(\nu,\zeta)\in \mc N(\nu+\mc X, \zeta)^1$.
  		\item[(3)] Let $E\in \mc F$. Then $E\otimes M(\nu,\zeta)$ has a filtration with subquotients $M(\nu+\eta,\zeta)$ with weights $\eta$ of $E$. 
  	\end{itemize}
\end{lem}
\begin{proof}
Part (1) is a consequence of \cite[Theorem 4.1, Lemma 4.3]{MS}; see also \cite[Theorem 5.3]{MS}. Part (2) is taken from \cite[Proposition 5.5]{MS}; see also \cite[Lamma 5.6, Proposition 5.8]{MS} for its proof. 	
	Part (3) is taken from \cite[Lemma 5.12]{MS}. 
\end{proof}

 \begin{proof}[Proof of the Theorem A]
 Let $\g$ be an arbitrary quasireductive Lie superalgebra. The  conclusion concerning the equivalence \eqref{eq::11} of Theorem A follows from Theorem \ref{thm::15}.  We are going to show the  Part (1) of Theorem A. 
 Recall that we have defined 
 \begin{align*}
 &S(\mu,\zeta)=\mc L(M(\nu),S_\nu(\mu))\otimes_{U(\g_\oa)} M(\nu,\zeta),
 \end{align*} for any $\zeta\in \ch \mf n_\oa^+$, $\nu \in \Lambda^+(\zeta)$ and $\mu\in \mc X(\nu).$ 
 By Theorem \ref{thm::15}, it follows that the set 
 \begin{align}
 &\{S(\mu,\zeta)|~\zeta\in \ch \mf n_\oa^+,~\nu\in \Lambda^+(\zeta),~\mu\in \mc X(\nu)\}, \label{eq::list511}
 \end{align}
 is a   list  of mutually non-isomorphic simple Whittaker modules  in $\widetilde{\mc N}(\zeta)$. It remains to show that every simple Whittaker module of integral-type lies in  the following set   
 \begin{align}
 &\{S(\mu,\zeta)|~\zeta\in \ch \mf n_\oa^+,~\nu\in \Lambda^+(\zeta)\cap \mc X,~\mu\in \mc X(\nu)\}. \label{eq::list5112}
 \end{align}
	To see this, we let $S$ be an integral type simple Whittaker module over $\g$ with a $\g_\oa$-submodule $L(\mu,\zeta)\hookrightarrow \Res S$, for some $\mu \in \mc X$ and $\zeta\in \ch \mf n_\oa^+$. Then  $S$ is a quotient of $\Ind L(\mu,\zeta)$.   Let  $\nu\in \mc X$ be dominant such that $W_\nu=W_\zeta$.  By  Parts (1) and (2) in Lemma \ref{lem::19}, it follows that $E\otimes M(\nu,\zeta)\in \mc N(\mc X, \zeta)^1$, for any $E\in \mc F$. In particular, $M(\mu, \zeta)\in \mc N(\mc X,\zeta)^1$ by Part (3) of Lemma \ref{lem::19}.  Therefore, we have $L(\mu,\zeta) \in \mc N(\mc X,\zeta)^1$, which implies that $\Res \Ind L(\mu,\zeta) \cong U(\g_\ob)\otimes L(\mu,\zeta)\in \mc N(\mc X,\zeta)^1$ by Part (1) of Lemma \ref{lem::19}. It follows that $\Ind L(\mu,\zeta) \in \widetilde{\mc N}(\mc X,\zeta)^1$. Consequently, we have  $S\in \widetilde{\mc N}(\mc X,\zeta)^1$.  The  conclusion in Part (1)  of Theorem A follows.

  Next,  assume that $\g$ is either a basic classical Lie superalgebra   or a Lie superalgebra of type I-0 with $\nu$  integral.  The   conclusion in Part (2) of Theorem A  follows from Theorem \ref{thm::realization}  and Lemma \ref{lem::15}.   This completes the proof. 
 \end{proof}

We  now turn to the proof of Theorem B in Section  \ref{thm::thmA}. Assume that $\g$ is of type I-0 with $W_\zeta$-anti-dominant $\la \in \mc X$. We continue to assume  that $\nu\in \mc X$ is dominant and  $W_\nu =W_\zeta$. 
\begin{proof}[Proof of the Theorem B]  
	
	By 
	 Lemma \ref{lem::15} we have  ${\Ann}_{U(\g)} S_\nu(\la) ={\Ann}_{U(\g)} \widetilde L(\la).$ It follows from     Theorem \ref{thm::15} that 
	 \begin{align}
	 &\text{Ann}_{U(\g)} S_\nu(\la) =  \text{Ann}_{U(\g)} \widetilde L(\la,\zeta).
	 \end{align}  	This completes the proof of Theorem B. 
\end{proof}

\begin{rem}  Let $\widetilde{T}_1$ (resp.  $T_1$) be the equivalence $\mc L(M(\nu), -)\otimes_{U(\g_\oa)}M(\nu,\zeta)$ from $\mc O^{\vpre}$ (resp. $\mc O_\oa^{\vpre}$) to $\widetilde{\mc N}(\nu+\mc X,\zeta)^1$ (resp. $\mc N(\nu+\mc X,\zeta)^1$). By Theorem \ref{thm::thm9} and Theorem \ref{thm::15}, we obtain the following commutative diagram
$$\xymatrixcolsep{3pc} \xymatrix{
	  \mc O^{\vpre} \ar[r]^-{\widetilde T_1} \ar@<-2pt>[d]_{\Res}  &   \widetilde{\mc N}(\nu+\mc X, \zeta)^1 \ar@<-2pt>[d]_{\Res} ,\\
  \mc O^{\vpre}_\oa \ar[r]^-{T_1} \ar@<-2pt>[u]_{\Ind} &    \mc N(\nu+\mc X, \zeta)^1\ar@<-2pt>[u]_{\Ind}.}$$
Namely, the equivalences $\widetilde{T}_1$ and $T_1$ intertwine the Frobenius extensions of the categories above. 
\end{rem}

\subsection{Some consequences} \label{Sect::cons}
We will provide in this section some consequence and complete the proof of Part (1) of Theorem C. Fix $\zeta \in \ch \mf n_\oa^+$.  

 \subsubsection{Simplicity of Kac-Whittaker modules}The following corollary gives a characterization of the simplicity of Kac modules in $\widetilde{\mc N}$ for  $\g= \gl(m|n),\mf{osp}(2|2n).$
\begin{cor} Consider $\mf g=\gl(m|n),\mf{osp}(2|2n)$. 
Let $\la\in \mc X$ and $\zeta \in \ch \n_\oa^+$. Then the
 Kac module	$K(\widetilde{L}(\la,\zeta))$ is simple if and only if $\la$ is typical {\em(}see Section \ref{sect::614} for the definitions of typicality of weights{\em)}.
\end{cor} 
\begin{proof}
	The conclusion follows from \cite[Corollary 6.8]{CM} and Theorem B.
\end{proof}

\subsubsection{Connection with projective-injective modules} Let $$R:=\bigoplus_{\widetilde P(\la),\widetilde P(\mu):\text{ injective}.}\Hom_{\mc O}(\widetilde P(\la), \widetilde{P}(\mu)).$$ For $\g=\gl(m|n)$, $\zeta$ non-singular and $\nu$ integral, it is established by Brundan, Losev and Webster \cite{BLW} that  the category $\text{mof-}R$ consisting of finite dimensional locally unital (right) modules over $R$ is  equivalent to a sum of certain abelian categories $\text{mod-}H_\xi$ ($\xi\in \mc X$) consisting of sequences of finite-dimensional modules over an infinite   {\em tower  cyclotomic quiver Hecke algebras} $H_\xi^1\subset H_\xi^2\subset H_\xi^3\subset \cdots$ subject to some stability conditions; see \cite[Lemma 4.6]{Br} and \cite[Theorem 4.9]{BLW}.  Theorem A and the following corollary allow to   make connections between   the various  abelian categories $\text{mof-}R$, $\mc O^{\vpre}$, $\mc B_\nu$ and $\widetilde{\mc N}(\nu+\mc X, \zeta)^1$.

\begin{cor}	Suppose that $\g$ is either a basic classical Lie superalgebra or a Lie superalgebra of type I-0 with non-singular $\zeta$ (i.e., $W_\zeta =W$). Let $\nu \in \mc X$ be  dominant such that  $W_\nu =W$.  Then $\widetilde{\mc N}(\nu+\mc X,\zeta)^1$ is  equivalent to $\text{mof-}R$.
\end{cor}
\begin{proof}
	The conclusion follows from Theorem \ref{thm::15} and \cite[Theorem 4.11]{CCC}.
\end{proof}

\subsubsection{A conjecture of Batra and Mazorchuk}\label{sec::533}

Suppose that $\g=\g_\oa$.  Batra and Mazorchuk made in  \cite[Conjecture 40]{BM}  a conjecture that for generic $\zeta \in \mf n_\oa^+$ and anti-dominant $\la$, the module $L(\la,\zeta)$  is the unique simple Whittaker module in $\mc N(\zeta)$, whose annihilator coincide with $\Ann_{U(\mf g)}M(\la)$. We prove the validity of this conjecture for the case of integral weights $\la.$

\begin{cor}
Retain the notations and assumptions above. Then the Batra-Mazorchuk conjecture \cite[Conjecture 40]{BM} is true in the case when $\la$ is integral. 
\end{cor}
\begin{proof} Let $\mu\in \h^\ast$ and set $\mu_0\in W_\zeta\cdot \mu$ to be $W_\zeta$-anti-dominant.  
   Suppose that  $\Ann_{U(\mf g)}L(\mu, \zeta) = \Ann_{U(\mf g)}M(\la)$. Then $\mu\in W\cdot \la$ by \cite[Proposition 2.1]{MS} and so $\mu$ is integral. It then follows from Theorem B that $\Ann_{U(\mf g)}L(\mu, \zeta) = \Ann_{U(\mf g)}L(\mu_0)$. Consequently, we have $\la=\mu_0$ and so $L(\mu,\zeta) = L(\la,\zeta)$. 
\end{proof}

We remark that the simple Whittaker modules studied in \cite[Section 4.2]{BM} are isomorphic to our $L(\la,\zeta)$; see also  \cite[Theorem 20]{Ch21}. 

\subsubsection{The proof of Part (1) of Theorem C}
For a given dominant weight $\nu\in \mc X$ such that $W_\nu =W_\zeta$, we recall the notation $\ups :=\ups_\nu$  from Section \ref{Sect::35}. For any $\la\in \mc X(\nu)$, it follows from Theorem \ref{thm::15} that    $$\widetilde{P}(\la,\zeta):=\mc L(M(\nu), \widetilde P(\la))\otimes_{U(\g_\oa)} M(\nu,\zeta),$$ is the projective cover of $S(\la,\zeta)$ in $\widetilde{\mc N}(\mc X,\zeta)^1$. The conclusion of Part (1) in Theorem C is a consequence of the following result. 
\begin{thm} \label{thm::24}
	Suppose that $\g$ is an arbitrary quasireductive Lie superalgebra. Let $\nu \in \mc X$ be dominant with $W_\nu =W_\zeta$. Define a functor $F:\mc O\rightarrow \mc O^{\vpre}$ by letting 
	$$F(-):=\mc L(M(\nu), -) \otimes_{U(\g_\oa)} M(\nu): \mc O\rightarrow \mc B_\nu \xrightarrow{\cong} \mc O^{\vpre}.$$ Then  
	\begin{itemize}
		\item[(1)] $F(-)$ is an exact functor.
		\item[(2)] 
		We have 
		\begin{align}
		&F(\wdL(\mu))\cong\begin{cases} S_\nu(\mu) &\mbox{ if $\mu \in \mc X(\nu)$,}\\
		0&\mbox{ otherwise.}
		\end{cases} \label{eq::twoeq}
		\end{align}
		\item[(3)] (Multiplicity formula \eqref{eq::mul} in Theorem C) For any $\la,\mu\in \mc X(\nu)$, we have \begin{align}
		&[\widetilde P(\la,\zeta):S(\mu,\zeta)] = [\widetilde P(\la): \wdL(\mu)]. 
		\end{align}   
	\end{itemize}
\end{thm}
\begin{proof}
	We first show Part (1).  Since $\mc L(M(\nu),-)$ is exact when restricted to $\mc O_\oa$ due to the projectivity of $M(\nu)$ by \cite[Lemma 6.2(i)]{BG}, the exactness of $\mc L(M(\nu),-): \mc O\rightarrow \mc B_\nu$ follows. 
	However, we give a self-contained proof as follows. For any $Y\in \mc B_\nu$, we denote by $Y^{\text{ad}}$ the (left) $\g_\oa$-module of $Y$ induced by the adjoint action of $\g_\oa$ that is given by $x\cdot y =xy-yx$, for any $x\in \g_\oa$, $y\in Y$. We may observe that the following two statements are equivalent:
	\begin{itemize}
		\item[(a)] The sequence 
		\begin{align*}
		&0\rightarrow \mc L(M(\nu), A)^{\text{ad}} \rightarrow \mc L(M(\nu), B)^{\text{ad}} \rightarrow \mc L(M(\nu), C)^{\text{ad}}  \rightarrow 0, 
		\end{align*} is exact for any short exact sequence $0\rightarrow A \rightarrow B\rightarrow C\rightarrow 0$ in $\mc O$.
		\item[(b)] $\Hom_{\mc O_\oa}(E, \mc L(M(\nu),-)^{\text{ad}})$ is an exact functor on $\mc O$, for any  $E\in \mc F$. 
	\end{itemize} 
On the other hand, as observed in the proof of \cite[Lemma 6.2(i)]{BG} for any $E\in \mc F$, we have  $$\Hom_{\mc O_\oa}(E, \mc L(M(\nu),-)^{\text{ad}})\cong \Hom_{\mc O_\oa}(E\otimes M(\nu),-)$$ by \cite[Lemma 2.2, Lemma 6.1]{BG}. The latter is exact since $E\otimes M(\nu)$ is projective in $\mc O_\oa$. This proves Part (1). 

Next, we prove Part (2).   Let $\mu \in \mc X(\nu)$. 
By the definition of $S_\nu(\mu)$, we have  a short exact sequence \begin{align*} 
&0\rightarrow U_{\ups}(\mu)\rightarrow  S_\nu(\mu) \rightarrow \wdL(\mu) \rightarrow 0, \end{align*}
where  $U_\ups(\mu)$ is the largest quotient of $\rad \widetilde P(\mu)$ such that  every composition factor $\wdL(\gamma)$ of $U_\ups(\mu)$ is not $\ups$-free, namely, $\wdL(\gamma)$ is $\alpha$-finite for some $\alpha \in \ups$.  By Part (1), we obtain a short exact sequence
\begin{align}
&0\rightarrow \mc L(M(\nu), U_\ups(\mu)) \rightarrow \mc L(M(\nu), S_\nu(\mu)) \rightarrow \mc L(M(\nu), \wdL(\mu))  \rightarrow 0.  
\end{align} By Lemma \ref{thm::14}, it remains to show that $\mc L(M(\nu), U_\ups(\mu))=0.$ Equivalently, we need to show that $\Hom_{\mc O_\oa}(M(\nu), E\otimes \Res U_\ups(\mu))=0$, for any $E\in \mc F$. 

Suppose on the contrary that there is $E\in \mc F$ such that $$\Hom_{\mc O_\oa}(M(\nu), E\otimes \Res U_\ups(\mu))\neq 0.$$
Then $[E\otimes \Res U_\ups(\mu): L(\nu)]\neq 0$, which implies that $[E\otimes \Res\wdL(\gamma): L(\nu)]\neq 0$, for some composition factor $\wdL(\gamma)$ of $U_\ups(\mu)$. We have already known that $\wdL(\gamma)$ is $\alpha$-finite for some $\alpha \in \ups$. This means that all composition factors of $E\otimes \Res \wdL(\gamma)$ are $\alpha$-finite, including $L(\nu)$, a contradiction to Lemma  \ref{lem::3}. Consequently, we have $\mc L(M(\nu), S_\nu(\mu)) \cong  \mc L(M(\nu), \wdL(\mu))$ and $S_\nu(\mu)\cong F(S_\nu(\mu))\cong F(\wdL(\mu))$.

 Now, we let $\mu \notin \mc X(\nu)$. If $\mu\notin \nu+\mc X$ then it follows directly  that $\Hom_{\mc O_\oa}(M(\nu), E\otimes \Res\wdL(\mu))=0$ for any $E\in \mc F$ and so $\mc L(M(\nu), \wdL(\mu))=0$. In the case $\nu \in \nu+\mc X$, the fact that $\Hom_{\mc O_\oa}(M(\nu), E\otimes \Res\wdL(\mu))=0$ for every $E\in \mc F$ can be proved by a similar argument. Therefore, we have  $\mc L(M(\nu), \wdL(\mu))=0$ again. This proves Part (2). 

Finally, since  $F(S_\nu(\mu))$ and $S_\nu(\mu)$ are isomorphic  by Lemma \ref{lem::115}, it follows that $F(\wdL(\mu_1))$ and $F(\wdL(\mu_2))$ are not isomorphic  provided that $\mu_1\neq \mu_2$, for any $\mu_1, \mu_2\in \mc X(\nu)$. In addition, we note that $F(\widetilde P(\mu))$ and $\widetilde P(\mu)$ are isomorphic   by Lemma \ref{lem::115} again. Therefore, Part (3) is a   consequence of Parts (1), (2) and Theorem \ref{thm::15}. This completes the proof. 


\end{proof}

  Using the functor $F(-)$ from Theorem \ref{thm::24}, we are able to provide a generalization of Lemma \ref{lem::15} as follows. We will use the notation  $\text{LAnn}_{U(\g)}Y$ to denote the left annihilator ideal of $Y$, for any $Y\in \mc B_\nu.$
 	
 	\begin{cor} \label{cor::306} Consider $\g$ an arbitrary quasireductive Lie superalgebra with $\nu\in \mc X$  dominant. Then we have 
 		\begin{align}
 		&\emph{\Ann}_{U(\g)}S_\nu(\mu,\zeta)=\emph{\Ann}_{U(\g)}S_\nu(\mu) = \emph{\Ann}_{U(\g)} \wdL(\mu),
 		\end{align} for any $\mu \in \mc X(\nu)$.
 	\end{cor}
 	\begin{proof} Fix $\mu\in \mc X(\nu)$. By definition, we have ${\Ann}_{U(\g)}S_\nu(\mu) \subseteq {\Ann}_{U(\g)} \wdL(\mu)$ since  $\wdL(\mu)$ is a quotient of $S_\nu(\mu)$. 	
 		
 		By Theorem \ref{thm::15}, we have 
 		\begin{align}
 		&{\Ann}_{U(\g)}S_\nu(\mu)=\text{LAnn}_{U(\g)}\mc L(M(\nu), S_\nu(\mu))={\Ann}_{U(\g)}S_\nu(\mu,\zeta).
 		\end{align} 
 		By the proof of Part (2) of Theorem \ref{thm::24}, we have $\mc L(M(\nu), S_\nu(\mu)) \cong \mc L(M(\nu), \wdL(\mu))$. It follows that  
 		\begin{align}
 		&{\Ann}_{U(\g)}S_\nu(\mu) \supseteq {\Ann}_{U(\g)} \wdL(\mu).
 		\end{align} The conclusion follows. 
 	\end{proof}
 	
 	\begin{cor} \label{cor::27}
 		Consider $\g$ a quasireductive Lie superalgebra with   $\nu\in \mc X$ dominant and $\alpha\in \ups_\nu$. Assume that either of the following holds:
 		\begin{itemize}
 			\item[$\bullet$] $\g$ is a Lie superalgebra of type I-0.
 			\item[$\bullet$] $\g$ is a basic classical Lie superalgebra such that $\alpha$ is a simple root with respect to the fixed triangular decomposition \eqref{eq::trian}. 
 		\end{itemize} 
 		
 		Then we have  $$F(\wdM(\mu))\cong F(\wdM(s_\alpha\cdot \mu)),$$
 		provided that $\wdL(\mu)$ is $\alpha$-finite. 
 		
 	\end{cor}
 	\begin{proof} It suffices to show that 
 		\begin{align}
 		&\mc L(M(\nu), \wdM(\mu)) \cong \mc L(M(\nu), \wdM(s_\alpha\cdot \mu)).
 		\end{align} 
 	Set $s:=s_\alpha$. As shown in the proof of \cite[Lemma 6.5 (2)]{Co16}, there is a four term exact sequence 
 	\begin{align}
 	&0\rightarrow \wdM(s\cdot \mu)\xrightarrow{f} \wdM(\mu)\rightarrow T_{s}\wdM(s\cdot\mu)\xrightarrow{g}  \wdM(s\cdot \mu) \rightarrow 0. \label{eq::523}
 	\end{align}
 	By Theorem \ref{thm::realization}, $T_{s}\wdM(s\cdot\mu)$ has a quotient $\wdM(s\cdot\mu)$ with a kernel that has no any $\alpha$-free composition factors. 	Since $\mc L(M(\nu),-)$ is exact, it follows that  
 	\begin{align*}
 	&\mc L(M(\nu), g): \mc L(M(\nu), T_{s}\wdM(s\cdot\mu)) \xrightarrow{\cong} \mc L(M(\nu), \wdM(s\cdot \mu)).
 	\end{align*}
  Consequently, we have \[\mc L(M(\nu), f):\mc L(M(\nu), \wdM(s\cdot \mu))\xrightarrow{\cong } \mc L(M(\nu), \wdM(\mu)).\]

  In the case of Lie superalgebras of type I, we have the four term exact sequence in \eqref{eq::523} by applying the Kac functor $K(-)$ to the corresponding four term exact sequence of $\g_\oa$-modules as given in \cite[Section 6.5]{AL}. Therefore, the conclusion for this case can be proved in the same way. This completes the proof.  
 	\end{proof}


\section{Linkage principle of the category $\widetilde{\mc N}^\Z$ for $\gl(m|n)$, $\mf{osp}(2|2n)$ and $\pn$
} \label{Sect::5Link}
 
 Fix $\zeta \in \ch \mf n^+_\oa$. Suppose that $\g$ is either a basic classical Lie superalgebra or a Lie superalgebra of type I-0 with $\nu\in \mc X$ dominant and $W_\nu=W_\zeta$.  Recall that we denote by $\ups_\nu$ the subset of roots $\alpha\in \Pi_0$ that satisfy $\langle \nu+\rho_\oa,\alpha^\vee \rangle =0$; see Section \ref{Sect::35}. As noted in  Part (1) of Theorem C (and its proof from Theorem \ref{thm::24}), one can construct linkages of integral type simple Whittaker modules in the category $\widetilde{\mc N} $ by computing $\ups_\nu$-free composition factors of projective covers that have $\ups_\nu$-free tops.  In this section, we  focus on the description of (indecomposable) blocks of $\widetilde{\mc N}$ over Lie superalgebras of type I. 
 
  Let $\g$ be reductive Lie algebras, $\gl(m|n),$ $\mf{osp}(2|2n)$, or $\pn$. Recall that  $\widetilde{\mc N}^\Z$ denotes the Serre  subcategory of $\widetilde{\mc N}$ generated by  simple Whittaker modules of integral type,  namely, $\widetilde{\mc N}^\Z$ consists of modules $M\in \widetilde{\mc N}$ such that $[M:\wdL(\la,\zeta)]=0$ unless $\la \in \mc X$. Similarly, we define $\mc{N}^\Z\subset \mc N$. We will describe the blocks of  $\widetilde{\mc N}^\Z$. Our main tools are central characters and linkages from Part (1) of Theorem C.
  

\subsection{Lie superalgebras $\gl(m|n)$, $\mf{osp}(2|2n)$ and $\pn$} \label{sect::26}
In this subsection, we introduce the type-I Lie superalgebras  $\gl(m|n)$, $\mf{osp}(2|2n)$ and $\pn$ given in Example \ref{ex::ex1}; see also \cite{ChWa12} and \cite{Mu12} for more details.  	

\subsubsection{The general linear Lie superalgebras $\gl(m|n)$} \label{sect::271}
For positive integers $m,n$, the general linear Lie superalgebra $\mathfrak{gl}(m|n)$ 
can be realized as the space of $(m+n) \times (m+n)$ complex matrices
\begin{align} \label{gllrealization}
\left( \begin{array}{cc} A & B\\
C & D\\
\end{array} \right),
\end{align}
where $A,B,C$ and $D$ are $m\times m, m\times n, n\times m, n\times n$ matrices,
respectively. The Lie bracket of $\gl(m|n)$ is given by the super commutator. 
Let $E_{ab}$, for $1\leq a,b \leq m+n$ be the elementary matrix in $\mathfrak{gl}(m|n)$. 
Its $(a,b)$-entry  is equal to $1$ and all other entries are $0$.

The type-I gradation of $\gl(m|n)$ is determined by  letting  
\begin{align*}
\mf{gl}(m|n)_1:=
\{\begin{pmatrix}
0 & B \\
0 & 0
\end{pmatrix}\}\quad\mbox{and}\quad \gl(m|n)_{-1}:=
\{\begin{pmatrix}
0 & 0 \\
C & 0
\end{pmatrix}\}.
\end{align*}  

We  fix the Borel subalgebra $\fb_{\oa}$ of $\gl(m|n)_{\oa}=\mathfrak{gl}(m)\oplus \gl(n)$
consisting of matrices in \eqref{gllrealization} with $B=C=0$ and $A,D$ upper triangular. 
The standard Cartan subalgebra $\mf h  $ consists of diagonal matrices. Denote  the dual basis of $\mf h^*$ by $\{\vare_1, \vare_2, \ldots, \vare_{m+n}\}$ with respect to the standard basis of $\mf h$  
\begin{align}
\{H_i:=E_{i,i}|~1\leq i \leq m+n \}\subset \gl(m|n). \label{eq::cartangl}
\end{align}  In particular, we have 
\begin{align} 
&\Phi =\{\epsilon_i-\epsilon_j\,|\, 1\le i\neq j\le m+n\}, \\
&\Phi_\oa^+=\{\epsilon_i-\epsilon_j\,|\, 1\le i<j\le m\}\cup \{\epsilon_i-\epsilon_j\,|\, m+1\le i<j\le m+n\}, \\
&\Pi_0=\{\epsilon_i-\epsilon_{i+1}\,|\, 1\le i\le m-1\}\cup \{\epsilon_i-\epsilon_{i+1}\,|\, m+1\le i\le m+n-1\}.
\end{align}  
The Weyl group $W$ is isomorphic to the symmetric group $\mf S_{m+n}$ on $m+n$ letters. We fix a non-degenerate $W$-invariant bilinear form $\langle\cdot, \cdot\rangle: \mf h^*\times \mf h^* \rightarrow \C$ by letting   $\langle\vare_i, \vare_j\rangle =\delta_{ij}$, for all $1\leq i,j\leq m$, and $\langle\vare_i, \vare_j\rangle =-\delta_{ij}$, for all $m+1\leq i,j\leq m+n$.

\subsubsection{The orthosymplectic Lie superalgebras $\mf{osp}(2|2n)$} \label{sect::272}
For a positive integer $n$, the matrix realization of the orthosymplectic Lie superalgebra $\mf{osp}(2|2n)$ inside $\mf{gl}(2|2n)$ is given by 
\begin{equation}
\mf{osp}(2\vert 2n)=
\left\{ \left( \begin{array}{cccc} c &0 & X &Y\\
0 & -c& V & U\\
-U^t& -Y^t & A &B\\
V^t &X^t & C& -A^t \\
\end{array} \right):
\begin{array}{c}
c\in \C;\,\, X,Y,V,U\in \C^{n};\\
A,B,C\in \C^{n^2};\\
B=B^t,\,\, C=C^t.
\end{array}
\right\}. \label{eq::osp}
\end{equation}
The type-I gradation of $\mf{osp}(2\vert 2n)$ is determined by    
\begin{equation*}
~~\mf{osp}(2\vert 2n)_{-1}=
\left\{ \left( \begin{array}{cccc} 0 &0 & 0 &0\\
0 & 0& V & U\\
-U^t& 0 & 0 &0\\
V^t &0 & 0& 0 \\
\end{array} \right)
\right\}, ~
\mf{osp}(2\vert 2n)_{1}=
\left\{ \left( \begin{array}{cccc} 0 &0 & X &Y\\
0 & 0& 0& 0\\
0& -Y^t & 0 &0\\
0 &X^t & 0& 0 \\
\end{array} \right)
\right\}.
\end{equation*}
We  fix the Borel subalgebra $\fb_{\oa}$ of $\mf{osp}(2|2n)_{\oa}\cong \C\oplus \mf{sp}(2n)$
consisting of matrices in \eqref{eq::osp} with $U=V=C=0$ and $A$ upper triangular. Again, the Cartan subalgebra $\mf h   $ consists of diagonal matrices in the realization \eqref{eq::osp}, and we realized $\mf h^\ast$ as a subspace of $\bigoplus_{i=1}^{2n+2}\C\vare_i$ inside the dual space of the Cartan subalgebra of $\gl(2|2n)$. This induces a bilinear form on $\mf h^\ast$. We note that the sets $\Phi_\oa^+$ and $\Pi_0$ can be identified as the root system and simple system of its subalgebra $\mf{sp}(2n)$. The Weyl group $W$ is isomorphic to the semi-direct product of $\Z_2^n$ and the symmetric group $\mf S_n$ on $n$ letters.

\subsubsection{The periplectic Lie superalgebra  $\pn$} \label{sect::273} 
Let $n$ be a positive integer. The matrix realization of the periplectic Lie superalgebra 
$\pn$ inside  
$\mathfrak{gl}(n|n)$ is given by
\begin{align}\label{plrealization}
\pn:=
\left\{ \left( \begin{array}{cc} A & B\\
C & -A^t\\
\end{array} \right)\| ~ A,B,C\in \C^{n\times n},~\text{$B^t=B$ and $C^t=-C$} \right\},
\end{align} where $B^t, C^t$ denote the transposes of $B,C$, respectively. 
  The type-I gradation of $\pn$ inherit that of $\gl(n|n)$, namely,  
\begin{align*}
\pn_1:=
\{\begin{pmatrix}
0 & B \\
0 & 0
\end{pmatrix}|B^t=B\}\quad\mbox{and}\quad \pn_{-1}:=
\{\begin{pmatrix}
0 & 0 \\
C & 0
\end{pmatrix}|C^t=-C\}.
\end{align*}

The standard Cartan subalgebra $\mf h$ consists of diagonal matrices. 
With slight abuse of notation, we again denote the dual basis of $\mf h^*$ by $\{\vare_1, \vare_2, \ldots, \vare_n\}$ with respect to the following standard basis of $\mf h$
\begin{align}
\{H_i:=E_{i,i}-E_{n+i,n+i}|~1\leq i \leq n \}\subset \pn. \label{eq::cartan}
\end{align}  
Then we have 
\begin{align}\label{eqroots}
&\Phi=\{\epsilon_i-\epsilon_j,~\pm(\vare_i+\vare_j)\,|\, 1\le i\neq j\le n\}\cup \{2\vare_i|~1\leq i\leq n\}, \\
&\Phi_\oa^+=\{\epsilon_i-\epsilon_j\,|\, 1\le i<j\le n\}, \\
&\Pi_0=\{\epsilon_i-\epsilon_{i+1}\,|\, 1\le i\le n-1\}.
\end{align} 
The Weyl group $W$ is isomorphic to the symmetric group $ \mf S_n$ on $n$ letters. Fix a non-degenerate $\mf S_n$-invariant bilinear form $\langle\cdot, \cdot\rangle: \mf h^*\times \mf h^* \rightarrow \C$ by $\langle\vare_i, \vare_j\rangle =\delta_{ij}$, for all $1\leq i,j\leq n$.  We  fix the Borel subalgebra $\fb_{\oa}$ of $\g_{\oa}=\mathfrak{gl}(n)$
consisting of matrices in \eqref{plrealization} with $B=C=0$ and
$A$ upper triangular. We put $\omega=\omega_n:=\vare_1+\vare_2+\cdots+\vare_n$. For any $c\in \C$, we note that  $\C_{c\omega} \otimes -$ provides  an auto-equivalence of $\g$-Mod. 

Since $x  \omega =\omega$ for any $x\in W$,  without loss of generality, we may shift the Weyl vector $\rho_\oa$ of $\g_\oa$ by letting  
$\rho_\oa:= (n-1)\vare_1 +(n-2)\vare_2+\cdots +\vare_{n-1}.$

\subsubsection{Typicality of weights} \label{sect::614}
For $\g:=\gl(m|n),~\mf{osp}(2|2n)$,   recall the notion of atypicality of weights from \cite[Section 2.2.6]{ChWa12}. A (odd) root $\alpha$ is called {\em isotropic} if $\langle \alpha ,\alpha\rangle=0$. Then $\la\in \h^\ast$ is said to be {\em atypical}, in case $\langle \la+\rho ,\alpha\rangle=0$, for some (odd) isotropic root $\alpha$, and {\em typical} otherwise.
  
For $\g = \pn$, the notion of typicality of weights for~$\mf{pe}(n)$ has been defined in~\cite[Section 5]{Se02}, but we will not use it in the present paper.

  \subsection{Technical tools for blocks} \label{Sect::Tetoolblocks}

For any $\la \in \h^\ast$, recall that we use $\chi_\la$ to denote the    central character of $\mf g$ associated to $\la$. The following lemma is useful in describing linkages in $\widetilde{\mc N}$.

\begin{lem} \label{Cor::chofWT}
	For any $\la\in \h^\ast$ and $\zeta \in \ch\mf n_\oa^+$, the   $\widetilde{L}(\la,\zeta)$ admits the central character $\chi_\la$.   
\end{lem}
\begin{proof}
	This is a direct consequence of Theorem B. We give an alternative proof as follows. It is shown in \cite[Section 2]{MS} that $\Ann_{U(\g_\oa)} M(\la,\zeta) =\Ann_{U(\g_\oa)} M(\la).$ Applying the Kac functor, we obtain that $\Ann_{U(\g)} \widetilde{M}(\la,\zeta) = \Ann_{U(\g)} \widetilde{M}(\la)$. The conclusion follows.
\end{proof}

For any $\la \in \h^\ast$, we define $\widetilde{\mc N}_{\chi_\la}$ to be the full subcategory of $\widetilde{\mc N}$ consisting of all modules $M$ such that $[M: \widetilde{L}(\mu ,\zeta)] =0$, for any $\mu \in \h^\ast$ with $\chi_\mu \neq \chi_\la$.  By Lemma \ref{Cor::chofWT}, we may observe that any short exact sequence in $\widetilde{\mc N}$
\begin{align}
&0\rightarrow \widetilde{L}(\la, \zeta) \rightarrow E \rightarrow \widetilde{L}(\mu,\zeta) \rightarrow 0
\end{align} splits unless $\chi_\la =\chi_\mu$. Since every module in $\widetilde{\mc N}$ has finite length,  we have the following decomposition by a standard argument (see, e.g., \cite[Lemma in Section 7.1]{Ja03})  
\begin{align}
&\widetilde{\mc N} =\bigoplus \widetilde{\mc N}_{\chi_\la},
\end{align} where the sum runs over all central characters  $\chi_\la$ of $\g$.

Following \cite[Chapter 7.1]{Ja03}, we consider the smallest equivalence on the set of simple modules in $\widetilde{\mc N}$ such that two simple modules $S,S' \in \widetilde{\mc N}$
are equivalent whenever there is a non-split short exact sequence 
\[0\rightarrow S \rightarrow E\rightarrow S' \rightarrow 0,\]
in $\widetilde{\mc N}$. The equivalence classes are called {\em blocks}; see also \cite[Section 1.13]{Hu08}. Therefore,  if two simple modules of $\widetilde{\mc N}$ lie in the same block then they admit the same central character.  The following lemma will be useful. 
 \begin{lem}\label{lem::16}
 	Let $\la, \mu \in \h^\ast$. If $[\wdM(\la,\zeta):\wdL(\mu, \zeta)]>0$ then $\wdL(\la,\zeta)$ and $\wdL(\mu,\zeta)$ are in the same block.
 \end{lem}
 \begin{proof}
 	By \cite[Theorem 9]{Ch21}, the standard Whittaker module $\wdM(\la,\zeta)$ has simple top $\wdL(\la,\zeta)$. The conclusion follows.
 \end{proof}
 
 
 The following useful lemma from  \cite{B} and \cite{Ch21} gives linkages for blocks by  computing composition factors of standard Whittaker modules. 
 
 \begin{lem} \label{thm::chthmC} \emph{(}\cite[Theorem C]{Ch21}, \cite[Theorem 6.2]{B}\emph{)} Let $\g$ be either a reductive Lie algebra or one of $\gl(m|n)$, $\mf{osp}(2|2n)$, $\pn$. For $\la,\mu \in \h^\ast$ and $\zeta\in \ch \mf n_\oa^+$,  we have 
 	\begin{align}
 	&[\widetilde{M}(\la, \zeta): \widetilde{L}(\mu, \zeta)] = \sum_{\gamma}[\widetilde{M}(\la): \widetilde{L}(\gamma)], \label{lem::28}
 	\end{align} where the summation runs over all $W_\zeta$-anti-dominant weights $\gamma$ such that $\mu\in W_\zeta\cdot \gamma$.
 \end{lem} 


\subsection{Block decomposition of ${\mc N}^\Z$}
Consider $\g=\g_\oa$ a reductive Lie algebra with a character  $\zeta \in \ch \n_\oa^+$. 
For any $\la, \mu \in \h^\ast$, it is known that 
\begin{align} & \chi_\la =\chi_\mu \Leftrightarrow \la \in W\cdot \mu, \label{eq::g0linkages}\end{align} see, e.g., \cite[Theorem 1.9]{Hu08}. 
Set $\mc O^\Z$ to be the full subcategory of $\mc O$ consisting of modules of integral weights. We note that $\mc O^\Z$ is a   subcategory of $ {\mc N}^\Z$. The set of  highest weights (of simple modules) in a block in $\mc O^\Z$ is the exact same as an equivalence class  linkage described in \eqref{eq::g0linkages}; see  \cite[Proposition 1.13]{Hu08}. We are now going to generalize this description to ${\mc N}^\Z$. Let $\Z\Phi_\oa^+$ denote the $\Z$-span of $\Phi_\oa^+$ (i.e. the root lattice of $\g_\oa$). Similarly, we define $\Z\Phi$.  The following lemma is a direct consequence of Lemma \ref{Cor::chofWT}. 
\begin{lem} \label{lem::18} Consider $\g =\g_\oa$ a reductive Lie algebra. 
	Let $\la, \mu\in \h^\ast$. If $L(\la,\zeta)$ and $L(\mu,\zeta)$ are in the same block, then $\mu\in \la +\Z\Phi_\oa^+$. In particular, every module $M\in  {\mc N}$ decomposes into a direct sum $M=M_{1}\oplus M_{2}$, where $M_1\in \mc N^\Z$ and $M_2\in \mc N$ such that $[M_2:L(\gamma,\zeta)]=0$ for any $\gamma\in \mc X.$
\end{lem}

We show in the following theorem that the blocks of $\mc N^\Z$ are still given by the linkage principal given in  \eqref{eq::g0linkages}. 
\begin{thm} \label{thm::Lielink}
Let $\la,\mu \in \mc X$. Then the simple modules ${L}(\la, \zeta)$ and ${L}(\mu, \zeta)$ are in the same  block of $\mc N$ if and only if $\la \in W\cdot \mu.$
\end{thm}
\begin{proof}
	We first note that if   ${L}(\la, \zeta)$ and ${L}(\mu, \zeta)$ are in the same   block of $\mc N$ then $\la \in W\cdot \mu$ by Lemma \ref{Cor::chofWT}.
	
	Conversely,   suppose that $\la \in W\cdot \mu.$ Let  $\la_0 \in W_\zeta\cdot \la$ and $\mu_0\in  W_\zeta\cdot \mu$  be   $W_\zeta$-anti-dominant. By Theorem  \ref{mainthm1typeI} we have $L(\la,\zeta) = L(\la_0,\zeta)$ and $L(\mu,\zeta) = L(\mu_0,\zeta)$. 
	Let $\eta\in W\cdot \la =W\cdot \mu$ be a dominant weight. 	 
	Then it follows from Lemma \ref{thm::chthmC} that
	\begin{align*}
	&[M(\eta,\zeta): L(\la_0,\zeta)] \geq [M(\eta): L(\la_0)] >0,
	\end{align*} which implies that $L(\eta, \zeta)$ and $L(\la_0, \zeta)$ are in the same block. Similarly, $L(\eta, \zeta)$ and $L(\mu_0,\zeta)$ are in the same block. This completes the proof. 
\end{proof}

\subsection{Block decomposition  of $\widetilde{\mc N}^\Z$}

\subsubsection{A separation lemma}
We will prove in this subsection that  $\widetilde{\mc N}^\Z$ is a summand of $\widetilde{\mc N}$. The following lemma will be useful. 

\begin{lem} \label{lem::21}
 Suppose that $\la, \mu \in \mc X$ such that $\la \in W\cdot \mu$.  Then  $\widetilde{L}(\la,\zeta)$ and $\widetilde{L}(\mu,\zeta)$ are in the same block. 
\end{lem}
\begin{proof} To see this, let
   \[0\rightarrow L(\alpha,\zeta) \rightarrow E \rightarrow  L(\beta,\zeta)\rightarrow 0\]
   be a non-split short exact sequence in $\mc N$, for some $\alpha,\beta \in \h^\ast$. Applying the Kac functor $K(-)$, by \cite[Theorem 51]{CCM} the socle of  $K(E)$ is simple, and hence we have   a non-split short exact sequence in $\widetilde{\mc N}$ 
    \[0\rightarrow K(L(\alpha,\zeta)) \rightarrow K(E) \rightarrow K(L(\beta,\zeta))\rightarrow 0.\] Since  $\widetilde{L}(\alpha, \zeta)$ and $\widetilde{L}(\beta, \zeta)$ are respectively simple quotients of  $ K(L(\alpha,\zeta))$ and $ K(L(\beta,\zeta))$, we may conclude that $\widetilde{L}(\alpha, \zeta)$ and $\widetilde{L}(\beta, \zeta)$ are in the same block. 
  
    Since every module of $\widetilde{\mc N}$ has finite length by \cite[Corollary 4]{Ch21}, the conclusion follows by Theorem \ref{thm::Lielink}.
\end{proof}

For a given subset $\Lambda\subseteq \mf h^\ast$, a module $M\in \mc N(\zeta)$ is said to be of type $\Lambda$ if every composition factor of  $M$ has the form $L(\gamma,\zeta)$  with $\gamma \in \Lambda$. By Lemma \ref{thm::chthmC}  the standard Whittaker module $M(\la, \zeta)$ is of type $\la +\Z\Phi_\oa^+$, for any $\la \in \h^\ast$. Also, for any $E\in \mc F$, if $\{\gamma_1, \gamma_2,\ldots,\gamma_\ell\}\subset \h^\ast$ is the set of   weights of $E$, then  by \cite[Lemma 5.12]{MS} the module $E\otimes M(\la, \zeta)$ is of type $\{\la+\gamma_1, \la+\gamma_2, \ldots, \la+\gamma_\ell\}+\Z\Phi_\oa^+$ and so is $E\otimes L(\la, \zeta)$. The following is an analogue of  Lemma \ref{lem::18}.

\begin{lem} \label{lem::23}
	Let $\la\in \mc X$ and $\mu \in \h^\ast$. If $\widetilde{L}(\la, \zeta)$ and $\widetilde{L}(\mu, \zeta)$ are in the same   block of $\widetilde{\mc N}$, then $\mu\in \la+\Z\Phi.$  Consequently, every module $M\in \widetilde{\mc N}$ decomposes into a direct sum $M=M_{1}\oplus M_{2}$, where $M_1\in \widetilde{\mc N}^\Z$ and $M_2\in \widetilde{\mc N}$ such that $[M_2:\wdL(\gamma,\zeta)]=0$ for any  $\gamma\in \mc X.$
\end{lem}
\begin{proof} To see this, it suffices to show that any short exact sequences in $\widetilde{\mc N}$ 
	\begin{align}
	&0\rightarrow \widetilde{L}(\beta, \zeta) \rightarrow E \rightarrow \widetilde{L}(\alpha, \zeta) \rightarrow 0,\label{eq1::lem22} \\
	&0\rightarrow \widetilde{L}(\alpha, \zeta) \rightarrow F \rightarrow \widetilde{L}(\beta, \zeta) \rightarrow 0,\label{eq2::lem22}
	\end{align}
	split, provided that $\alpha \in \mc X$ and $\beta \notin \alpha+\Z\Phi.$

	We first show that \eqref{eq1::lem22} is a split short exact sequence. Since $\Res\widetilde{L}(\alpha, \zeta)$ is a quotient of $\Res{K}(L(\alpha, \zeta))\cong U(\mf g_{-1})\otimes {L}(\alpha, \zeta)$, it is of type $\alpha+\Z\Phi$. For the same reason, the module $\Res\widetilde{L}(\beta, \zeta)$ is of type $\beta+\Z\Phi$.  Since $\beta\notin \alpha +\Z\Phi$, it follows from  Theorem \ref{thm::Lielink} that   the $\mf g_\oa$-module $\Res E$ decomposes $\Res E =\Res \widetilde{L}(\alpha, \zeta) \oplus B$, for some $B\in \mc N$ isomorphic to $\Res\widetilde{L}(\beta, \zeta)$. We note that  $U(\g_\ob)B$ is an epimorphic image of $\Ind B$, and so $U(\g_\ob)B$ is of type $\beta+\Z\Phi$ as a $\g_\oa$-module. By Theorem \ref{thm::Lielink}, the submodules $\widetilde{L}(\alpha, \zeta)$, $U(\g_\ob)B$ of $E$ have trivial intersection since $\alpha+\Z\Phi$ and $\beta+\Z\Phi$ have empty intersection. Consequently, $E$ is a direct sum of two non-trivial submodules in $\widetilde{\mc N}$, we may conclude that \eqref{eq1::lem22} splits. 
	
	Similarly, the short exact sequence \eqref{eq2::lem22} splits by the same argument.
	
	
\end{proof}

\subsubsection{Blocks of $\widetilde{\mc N}^\Z$ for $\gl(m|n)$ and $\mf{osp}(2|2n)$} \label{Sect::532}
In this subsection, we consider $\g=\gl(m|n),~\mf{osp}(2|2n).$
Recall that the equivalence $\sim$ on $\mc X$ from Section \ref{sec1} is defined by declaring $\la\sim \mu, \text{ if } \mu = w \cdot (\la-\sum_{i=1}^kc_i \alpha_i),$ for some $w\in W$, mutually orthogonal isotropic (odd) roots $\alpha_1, \alpha_2,\ldots, \alpha_k\in\Phi$ and integers $c_1, c_2,\ldots, c_k\in \Z$ such that $\langle\la+\rho, \alpha_i\rangle=0,$ for all $1\leq i\leq k.$ We are now in a position
to state our first main result of this section.

\begin{thm} \label{thm::blocksglosp}
	Let $\la \in \mc X$ and $\mu\in \h^\ast$. Then the simple modules $\widetilde{L}(\la, \zeta)$ and $\widetilde{L}(\mu, \zeta)$ are in the same   block of $\widetilde{\mc N}$ if and only if $\la \sim \mu.$
\end{thm}
\begin{proof}
First, we suppose that $\widetilde{L}(\la, \zeta)$ and $\widetilde{L}(\mu, \zeta)$ are in the same   block of $\widetilde{\mc N}$. It follows from Lemma \ref{lem::23} that  $\mu \in \mc X$. Also, we have $\chi_\la =\chi_\mu$. Therefore,  by  \cite[Theorem 2.30]{ChWa12} we have   $\mu = w \cdot (\la -\sum_{i=1}^kc_i \alpha_i)$ for some $w\in W$ and mutually orthogonal isotropic (odd) roots $\alpha_1, \alpha_2,\ldots, \alpha_k\in \Phi$ and complex numbers $c_1, c_2,\ldots, c_k\in \C$ such that $\langle\la+\rho, \alpha_i\rangle=0,$ for all $1\leq i\leq k$. We note that $\sum_{i=1}^kc_i \alpha_i \in \Z\Phi$ since $\la \in \mc X$ and $\mu \in \la+\Z\Phi$. Since $\alpha_1, \alpha_2,\ldots, \alpha_k$ are  mutually orthogonal, it follows that the numbers $c_1,c_2,\ldots, c_k$ are integers. Consequently, we have $\la \sim \mu.$

Next, we shall show the converse implication. We first consider the cases that $\g=\gl(m|n)$. By Lemma \ref{lem::21} it suffices to show that $\widetilde{L}(\la,\zeta)$ and $\widetilde{L}(\la-\alpha,\zeta)$ lie in the same block if $\langle\la+\rho, \alpha\rangle =0$, for an odd   root $\alpha=\vare_i-\vare_{m+j}$ with some $1\leq i\leq m,~1\leq j\leq n$. To see this, we let $\sigma \in W_\zeta$ such that $\sigma\cdot (\la-\alpha)$ is $W_\zeta$-anti-dominant. Note that  $\widetilde{L}(\la,\zeta)$ and  $\widetilde{L}(\sigma\cdot \la,\zeta)$ are in the same block by Lemma \ref{lem::21}.

Also, we note that $\sigma\cdot (\la -\alpha) =\sigma\cdot\la -\sigma \alpha$ and $\langle\sigma \cdot \la +\rho ,\sigma\alpha\rangle =\langle\sigma(\la+\rho),\sigma\alpha\rangle=0.$ Since $\sigma \alpha$ is still  root in $\g_1$, it follows that 
\begin{align}
&[\widetilde{M}(\sigma\cdot \la,\zeta): \widetilde{L}(\sigma\cdot \la -\sigma\alpha, \zeta)] \geq [\widetilde{M}(\sigma\cdot \la):\widetilde{L}(\sigma\cdot \la-\sigma\alpha)]>0,
\end{align} by the consequence of the super  Jantzen sum formula in \cite[Proposition 2.2]{CW18G3}; see also \cite{Gor4} and \cite[(10.3)]{Mu12}. Therefore, $\widetilde{L}(\sigma\cdot \la, \zeta)$
and $\widetilde{L}(\sigma\cdot (\la-\alpha), \zeta)$ are in the same block. Finally, by Lemma \ref{lem::21} again, we may conclude that  $\widetilde{L}(\sigma\cdot(\la-\alpha), \zeta)$ and $\widetilde{L}(\la-\alpha, \zeta)$ are in the same block. The cases that $\g=\mf{osp}(2|2n)$ can be proved in similar fashion. The conclusion follows.
\end{proof}

\subsubsection{Blocks of $\widetilde{\mc N}^\Z$ for $\pn$} \label{Sect::533}
We define the equivalence relation $\sim^{\pn}$ on $\mc X$  transitively generated by $\la\sim^{\pn} w\cdot \la$ and $\la\sim^{\pn} \la \pm2\vare_k$, for~$w\in W$ and $1\le k\le n$. With slight abuse of notation, we again use $\sim$ to denote $\sim^{\pn}$ in this subsection.
 
 We define the a dual version of Kac functor $$K'(-):=\Ind_{\g_{\leq 0}}^\g (-\otimes \Lambda^{\text{max}} \g_{1}^\ast) \cong \Coind_{\g_{\leq 0}}^\g (-):\mf g_\oa\mbox{-Mod} \rightarrow \mf g\mbox{-Mod}.$$

 \begin{lem}\label{lem::dualchthmC}
   Let $\la,\mu \in \h^\ast$ and $\zeta\in \ch \mf n_\oa^+$.
 Then any two composition factors of $K'({M}(\la, \zeta))$ are in the same block.
   
    Furthermore, we have 
 		\begin{align}
 		&[K'({M}(\la, \zeta)): \widetilde{L}(\mu, \zeta)] = \sum_{\gamma}[K'(M(\la)): \widetilde{L}(\gamma)],
 		\end{align} where the summation runs over all $W_\zeta$-anti-dominant weights $\gamma$ such that $\mu\in W_\zeta\cdot \gamma$.
 \end{lem}
\begin{proof} By using the same argument as in \cite[Theorem 9-(i)]{Ch21}, the $K'({M}(\la, \zeta))$ has a simple top. Then the first claim follows by an analogue of Lemma \ref{lem::16}. 
 	Next, we recall the functor $\widetilde{\Gamma}_\zeta$  from \cite[Section 5.2]{Ch21}, which is an exact functor  from  $\mc O$ to  $\widetilde{\mc N}$.  Using an argument similar to that used in the proof of \cite[Theorem 20]{Ch21}, one can show that $\widetilde{\Gamma}_\zeta$ sends $K'({M}(\la))$ to $K'({M}(\la, \zeta))$.  The conclusion follows from \cite[Theorem 20]{Ch21}.
\end{proof}

\begin{thm} \label{thm::LLLinkpn}
	Let $\la,\mu \in \mc X$. Then the simple modules $\widetilde{L}(\la, \zeta)$ and $\widetilde{L}(\mu, \zeta)$ are in the same   block of $\widetilde{\mc N}$  provided that $\la \sim \mu.$ In particular, there are at most $n+1$ blocks of $\widetilde{\mc N}^\Z$, up to equivalence.
\end{thm}
\begin{proof}
	By Lemma \ref{lem::21}, it remains to show that $\widetilde{L}(\la,\zeta)$ and $\widetilde{L}(\la-2\vare_k,\zeta)$ are in the same block for any $1\leq k\leq n$. To see this, we first assume that $\la$ is $W_\zeta$-anti-dominant.
	Suppose that $\la -2\vare_k$ is  $W_\zeta$-anti-dominant as well. 
	 By  Lemma \ref{lem::dualchthmC} we find that
	\begin{align}
	&[K'(M(\la,\zeta)): \widetilde{L}(\la -2\vare_k,\zeta)] \geq [K'(M(\la)): \widetilde{L}(\la -2\vare_k)]. 
	\end{align} Both $[K'(M(\la)): {L}(\la)]$ and $[K'(M(\la)): \widetilde{L}(\la -2\vare_k)]$ are positive by \cite[Lemma 5.2]{CC}. The conclusion follows.

	Next, we assume that $\la$ is $W_\zeta$-anti-dominant, but $\la -2\vare_k$ is not   $W_\zeta$-anti-dominant. Use Lemma \ref{lem::21} and replace $\wdL(\la -2\vare_k)$ by $\wdL(\la -\tau(2\vare_k))$ for some $\tau \in W_\zeta$ if necessary, we may either reduce to the previous case, or assume that there exists a  positive even root $\alpha$ such that $s_\alpha\cdot(\la-2\vare_k)$ is $W_\zeta$-anti-dominant and less than $\la-2\vare_k$ in   Bruhat order. We know that $\widetilde{L}(\la -2\vare_k, \zeta)$ and $\widetilde{L}(s_\alpha\cdot (\la -2\vare_k), \zeta)$ are in the same block by Lemma  \ref{lem::21}.

	 By Lemma \ref{lem::dualchthmC} and \cite[Lemma 5.2]{CC} we have 
	 \begin{align*}
	 &[K'(M(\la,\zeta)): \widetilde{L}(\la,\zeta)]\geq [K'(M(\la)): \widetilde{L}(\la)]>0,\\
	 &[K'(M(\la)): \widetilde{L}(\gamma)]\geq [\widetilde{M}(\la-2\vare_k)): \widetilde{L}(\gamma)],
	 \end{align*}   for any $\gamma \in \h^\ast.$ Since $[{M}(\la-2\vare_k): {L}(s_\alpha\cdot(\la -2\vare_k))]>0$, applying the Kac functor $K(-)$ we have  that  $[\widetilde{M}(\la-2\vare_k): \widetilde{L}(s_\alpha\cdot(\la -2\vare_k))]>0$. Consequently, it follows from Lemma \ref{lem::dualchthmC} that  
	\begin{align*}
	&[K'(M(\la,\zeta)): \widetilde{L}(s_\alpha\cdot(\la -2\vare_k),\zeta)] \\
	&\geq[K'(M(\la)): \widetilde{L}(s_\alpha\cdot(\la -2\vare_k))]\\
	&\geq [\widetilde{M}(\la-2\vare_k): \widetilde{L}(s_\alpha\cdot(\la -2\vare_k))]>0,
	\end{align*}
	which implies that $\widetilde{L}(s_\alpha\cdot(\la -2\vare_k), \zeta)$ and  $\widetilde{L}(\la, \zeta)$ are in the same block, and so $\widetilde{L}(\la, \zeta)$ and $\widetilde{L}(\la-2\vare_k, \zeta)$ are in the same block. 
	
  Finally, let $\la \in \mc X$ be arbitrary. Pick  $\sigma\in W$ such that $\sigma\cdot \la$ is $W_\zeta$-anti-dominant.   Then  $\widetilde{L}(\la, \zeta)$ and $\widetilde{L}(\sigma\cdot\la, \zeta)$ are in the same block by Lemma \ref{lem::21}. Let $1\leq k'\leq n$ such that $\sigma \vare_k =\vare_{k'}$. Since $\sigma\cdot \la -2\vare_{k'} =\sigma\cdot(\la -2\vare_k)$, the simple modules $\widetilde{L}(\sigma\cdot\la-2\vare_{k'},\zeta)$ and $\widetilde{L}(\la-2\vare_k,\zeta)$ are in the same block. Now, we put all results together to infer that   all simple modules  $\widetilde{L}(\la,\zeta)$,     $\widetilde{L}(\sigma\cdot\la,\zeta)$, $\widetilde{L}(\sigma\cdot\la-2\vare_{k'},\zeta)$ and  $\widetilde{L}(\la-2\vare_k,\zeta)$ are in the same block. This proves the first conclusion.
  
  Recall that $\omega:=\vare_1+\vare_2+\cdots+\vare_n$.   We note that, for any $k\in \C$, the functor $\C_{k\omega}\otimes -$ is an auto-equivalence on $\widetilde{\mc N}$ sending $\wdL(\la,\zeta)$ to $\wdL(\la+k\omega,\zeta)$, for any $\la\in \h^\ast$. The conclusion follows from \cite[Section 5.2]{CC}.
\end{proof}

{\bf Question}: Is it true that $\la \sim \mu$ whenever  $\widetilde{L}(\la, \zeta)$ and $\widetilde{L}(\mu, \zeta)$ are in the same  block of $\widetilde{\mc N}$?

 \subsubsection{An example of non-type I Lie superalgebra} \label{cor::644}
 In this section, we illustrate the utility of Theorem C with an example of a block for $\g=\mf{osp}(3|2)$ in which neither of the constructions for simple Whittaker modules in Cases 1 and 2 of Section \ref{sect::12} are applicable. 
 
 We use notations in \cite[Section 1.2]{ChWa12}. The dual space of the Cartan subalgebra is given by $\h^\ast =\C\delta \oplus \C\vare.$ The simple system, positive system, the sets of roots and positive  even roots are given by
 \begin{align}
 &\Pi = \{ \delta -\vare, \vare\},~
 \Phi^+=\{\delta\pm \vare,2\delta,\delta,\vare\},~\Phi=\Phi^+\cup(-\Phi^+),~ 
 \Phi^+_\oa = \{2\delta,\vare\}.
 \end{align} The corresponding Weyl group $W= \langle s_{2\delta}, s_\vare \rangle$ is isomorphic to $\Z_2\times \Z_2$.  
 Let $\rho = \frac{-1}{2}\delta+\frac{1}{2}\vare$ denote the corresponding Weyl vector. The usual bilinear form $\langle,\rangle$ on $\h^\ast$ is determined by   $\langle\delta,\vare\rangle=0$, $\langle\delta,\delta\rangle=-1$ and $\langle\vare,\vare\rangle=1.$
  
 Let us consider $\zeta \in \ch\mf n_\oa^+$ a non-singular character. Set $\nu\in \mc X$ such that $W_\nu=W_\zeta=W$.   By \cite[Section 4.3.4]{CCC}, we have 
 \[\mc X(\nu)=\{\la|~\la+\rho = a\delta+b\vare,~a,b\in \frac{-1}{2}-\Z_{\geq 0}\}.\]
 
 We consider the block of $\widetilde{\mc N}^\Z$ containing $S(\la,\zeta)$ for $\la\in \mc X(\nu)$. For any $\la, \mu\in \mc X(\nu)$, we note that $\wdL(\la)$ and $\wdL(\mu)$ are in the same block of $\mc O$ if and only if $\chi_\la=\chi_\mu$ by  \cite[Proposition 2.2]{CW18G3}.   A weight $\la$ is said to be typical if $\la+\rho=a\delta+b\vare$ such that $a\pm b\neq 0$, which is equivalent to the condition that $\langle \la+\rho, \alpha\rangle \neq 0$, for $\alpha = \delta\pm\vare$; see also \cite[Section 2]{ChWa12} for the definition of typicality of weights for general $\mf{osp}(2m|2n+1)$.
\begin{prop}
	Let $\la,\mu \in \mc X(\nu)$. Then $S(\la,\zeta)$ and $S(\mu,\zeta)$ are in the same block of $\widetilde{\mc N}$ if and only if $\chi_\la =\chi_\mu$, which is equivalent to the condition that $\la =w(\mu+k\alpha+\rho)-\rho$, for some $w\in W$, $k\in \Z$ and $\alpha \in \{\delta \pm \vare\}$ with $\langle \mu+\rho,\alpha\rangle=0.$
\end{prop}
\begin{proof} It remains to show that $S(\la,\zeta)$ and $S(\mu,\zeta)$ lie in the same block provided that $\chi_\la=\chi_\mu$. We first assume that $\la$ is typical. Since $\chi_\la=\chi_\mu$, it follows that $\la=w(\mu+\rho)-\rho$ for some $w\in W$, which implies that $w=e$. The conclusion follows.
	
	We now suppose that $\la$ is atypical, namely, $\la+\rho =k\delta+k\vare$ with $k\in \frac{-1}{2}-\Z_{\geq 0}$. By \cite[Proposition 2.2]{CW18G3}, which is a consequence of the super Jantzen sum formula established in \cite{Gor4} and \cite[(10.3)]{Mu12}, we may conclude that the tilting module of ($\rho$-shifted) highest weight $-\la$ contains a subquotient of Verma module of the ($\rho$-shifted) highest weight $-\la-\delta-\vare$.  It then follows from the Arkhipov-Soergel duality \cite{Br04,So98} for Lie superalgebras that the projective cover $\widetilde P(\la)$ has a subquotient isomorphic to the Verma module $\wdM(\la +\delta+\vare)$. By Part (1) of Theorem A, we may conclude that $[\widetilde P(\la,\zeta): S(\la + \delta +\vare,\zeta)]$ is positive. It follows that  $S(\la,\zeta)$ and $S(\la +\delta+\vare, \zeta)$ are in the same block. Consequently, $S(\la,\zeta)$ and $S(\frac{-1}{2}(\delta+\vare),\zeta)$ lie in the same block. This completes the proof. 
\end{proof}

\vspace{2mm}

\noindent
Chih-Whi Chen:~Department of Mathematics, National Central University, Zhongli District, Taoyuan City, Taiwan;
E-mail: {\tt cwchen@math.ncu.edu.tw}
\hspace{2cm}



\end{document}